\documentclass[twoside]{article}
\long\def\symbolfootnote[#1]#2{\begingroup%
\def\thefootnote{\fnsymbol{footnote}}\footnote[#1]{#2}\endgroup}

\usepackage{amsfonts, amssymb, amsmath, amsthm}
\usepackage[colorlinks=true, pdfstartview=FitV, linkcolor=blue,
            citecolor=blue, urlcolor=blue]{hyperref}
\usepackage{pifont}
\usepackage[usenames]{color}
\definecolor{Red}{rgb}{0.7,0,0.1}
\definecolor{Green}{rgb}{0,0.7,0}

\newcommand{\fix}[1]{\footnote{{\textcolor{Red}{{\bf  Cuidado:}}}}}

\newcommand{\indeed}[1]{}

\title{Invariant Measures for Dissipative Dynamical Systems:\\
Abstract Results and Applications\footnote{{\bf To appear in Communications in Mathematical Physics}}}
\author{Micka\"el D. Chekroun$^{\sharp}$ and Nathan E. Glatt-Holtz$^\flat$}

\date{}

\numberwithin{equation}{section}

\newtheorem{Thm}{Theorem}[section]
\newtheorem{Lem}{Lemma}[section]
\newtheorem{Prop}{Proposition}[section]

\newtheorem{Def}{Definition}[section]
\newtheorem{Rmk}{Remark}[section]

\newcommand{\pd}[1]{\partial_{#1}}

\newcommand{\dr}{{\rm d}}
\newcommand{\DO}{\mathcal{D}}
\newcommand{\R}{\mathbb{R}}
\newcommand{\LIM}{\underset{T \rightarrow \infty}{\rm{LIM}}}
\newcommand{\memK}{\mu}
\newcommand{\brlMsr}{\mathfrak{m}}
\newcommand{\Gadol}{\Lambda}

\providecommand{\MR}{\relax\ifhmode\unskip\space\fi MR }

\def\bt{\begin{thm}}
\def\et{\end{thm}}
\def\bl{\begin{lem}}
\def\el{\end{lem}}
\def\bd{\begin{defi}}
\def\ed{\end{defi}}
\def\bc{\begin{cor}}
\def\ec{\end{cor}}
\def\bp{\begin{proof}}
\def\ep{\end{proof}}
\def\br{\begin{rem}}
\def\er{\end{rem}}

\def\be{\begin{equation}}
\def\ee{\end{equation}}
\def\bes{\begin{equation*}}
\def\ees{\end{equation*}}
\def\bea{\begin{equation}  \begin{aligned}}
\def\eea{\end{aligned} \end{equation}}
\def\beas{\begin{equation*} \begin{aligned}}
\def\eeas{\end{aligned} \end{equation*}}

\begin{document}


\maketitle

\vskip-4mm

\centerline{\footnotesize{\it $^{\sharp}$ Department of Atmospheric Sciences and Institute of Geophysics and Planetary Physics}}
\vskip-1mm
\centerline{\footnotesize{\it University of California, Los Angeles, CA 90095-1565, USA}}
\vskip-1mm
\centerline{\footnotesize{\it and}}
\vskip-1mm
\centerline{\footnotesize{\it Environmental Research and Teaching Institute (CERES-ERTI)}}
\vskip-1mm
\centerline{\footnotesize{\it \'Ecole Normale Sup\'erieure, 75231 Paris Cedex 05, France}}
\vskip-1mm
\centerline{\footnotesize{\it email:} \url{mchekroun@atmos.ucla.edu}}

\vskip2mm

\centerline{\footnotesize{\it $^\flat$ Department of Mathematics and The Institute for Scientific Computing and Applied Mathematics} }
\vskip-1mm
\centerline{\footnotesize{\it Indiana University, Bloomington, IN 47405, USA}}
\vskip-1mm
\centerline{\footnotesize{\it email:} \url{negh@indiana.edu}}

\begin{center}
\large
\date{\today}
\end{center}

\vskip4mm

\begin{abstract}
In this work we study certain invariant measures that can
be associated to the time averaged observation of a broad class of dissipative
semigroups via the notion of a generalized Banach limit.
Consider an arbitrary complete separable metric space $X$
which is acted on by any  continuous semigroup $\{S(t)\}_{t \geq 0}$.
Suppose that $\{S(t)\}_{t \geq 0}$ possesses a global attractor $\mathcal{A}$.
We show that, for any generalized Banach limit $\LIM$ and any probability
distribution of initial conditions $\brlMsr_0$,  that there exists an invariant
probability measure
$\brlMsr$, whose support is contained in $\mathcal{A}$, such that
    \begin{displaymath}
       \int_{X} \varphi(x) \dr\brlMsr(x)
       = \LIM \frac{1}{T}\int_0^T \int_X \varphi(S(t) x) \dr\brlMsr_0(x) \dr t,
    \end{displaymath}
for all observables $\varphi$ living in a suitable function space of continuous mappings
on $X$.

    This work
     is based on the framework of \cite{FoiasManleyRosaTemam1}; it
     generalizes and simplifies the proofs of more recent works
\cite{Wang2009},\cite{LukaszewiczRealRobinson2011}.  In particular our results
rely on the novel use of a general but elementary topological observation, valid in any
metric space, which concerns the growth of continuous functions in the
neighborhood of compact sets.
In the case when $\{S(t)\}_{t \geq 0}$  does not possess a compact absorbing set, this lemma allows
us to sidestep the use of weak compactness arguments
which require the imposition of cumbersome weak continuity conditions and
thus restricts the phase space $X$ to the case of a reflexive Banach space.

    Two examples of concrete dynamical systems where the semigroup
    is known to be non-compact are examined in detail.  We first consider
    the Navier-Stokes equations
    with memory in the diffusion terms.  This is the so called Jeffery's model
    which describes certain classes of viscoelastic fluids.
    We then consider a family of neutral delay differential
    equations, that is equations with delays in the time derivative terms.  These
    systems may arise in the study of wave propagation problems coming from
    certain first order hyperbolic partial differential equations; for example for
    the study of line transmission problems.
  For the second example the phase space is $X = C([-\tau,0],\mathbb{R}^n)$, for some
  delay $\tau > 0$, so that $X$ is not reflexive in this case.
\end{abstract}

{\noindent
  }

\newpage

\section{Introduction}

Essential characteristics of dynamical systems are described by invariant probability measures.
These measures identify statistical equilibria and can provide important information about
the long time behavior of the dynamics.  It is therefore of paramount interest both in theory
and applications to better understand this class of probability measures.

Frequently however, statistically steady states represented by invariant measures
are difficult to determine.
In practice physicists
and engineers compute approximations of invariant probability measures
by taking averages of time series (observables) associated to the system and
invoking an `ergodicity assumption'; they posit an equivalence between
these temporal averages and averages against the unknown
invariant measure which they are trying to identify.
Mathematically speaking, the complete and rigorous justification
of such an `ergodic hypothesis' seems
unreachable for non-specific classes of dynamical systems and remains a challenging
problem even for specific examples.   With this backdrop in mind, our aim in this and future works
(\cite{ChekrounGlattHoltz2011a,ChekrounGlattHoltz2011b})
is to further the development of a particular mathematical framework coming from
\cite{FoiasManleyRosaTemam1}, which
establishes a weak link between ensemble averages and temporal averages based
on generalized limits but which we show is in fact applicable to a 
wide class of dissipative
dynamical systems.\footnote{ In this article, dissipative dynamical systems are defined as those 
which possess a bounded absorbing set, cf. \cite{Temam3}.  
Note that other authors, e.g. \cite{Hale1988}, have referred to dynamical systems 
with this property as having {\it bounded dissipation}.}

For the study of infinite dimensional evolution equations a number
of different mathematical approaches to ergodicity have been
developed, each relevant to different situations.  One approach has
been to focus on some classes of simple linear or semi-linear first order
partial differential equations.  See, for example, \cite{Lasota1979, Dawidowicz1983,
BrunovskyKomornik1984, Rudnicki1988, Rudnicki2004} and references
therein.  While these works are significant for the fact that
they provide examples of linear systems where solutions exhibit
chaotic or `turbulent' behavior, the methods developed are limited
to a very specific class of equations. Another,
different, approach involves the study of stochastic partial
differential equations (and other related infinite dimensional
stochastic systems) where ergodicity is defined in terms of the
Markov semigroup generated by the stochastic semiflow. See e.g.
\cite{ZabczykDaPrato2, FlandoliMaslowski1, KuksinShirikyan1,
DaPratoDebussche, HairerMattingly1, HairerMattingly2011,
Debussche2011a}.  The methods developed in these works apply to a
wide variety of stochastically perturbed dissipative linear and
nonlinear systems, for instance the Navier-Stokes equations,
Reaction-Diffusion equations, Delay Equations, weakly damped nonlinear
Schr\"{o}dinger equations, complex Ginzburg-Landau equations.
On the other hand, this approach relies in an essential way on an
underlying mechanism of stochasticity, one in which the noise
injected into the system has to take a very specific form. Furthermore, 
this approach identifies invariant measures of the Markov semigroup
which are
deterministic probability measures and thus are not  carried by the
 global random attractor.\footnote{On other hand these invariant measures
are linked to the global random attractor; see \cite{CSG11} for a discussion 
of such relationships in the context of stochastic differential equations
(in finite dimensions).}  

In this work we follow a different approach, pioneered in \cite{FoiasTemam1975} and developed
in \cite{BCFM95,FoiasManleyRosaTemam1}.
These works link ensemble and temporal averages via the notion of the so called `generalized Banach limit',
a linear functional acting over the space of bounded continuous functions which in particular associates those elements
converging at $+ \infty$ with their classical limit; cf. Definition~\ref{def:GBL} below.
While the ideas in these works were developed in the specific setting of the $2D$ Navier-Stokes equations
on a bounded domain some essential aspects of the framework generalize in a straightforward way to
any compact semigroup $\{ S(t)\}_{t \geq 0}$ acting on a complete, separable metric space $(X,d)$.
On the other hand, in the important
case of non-compact semigroups (see below for an extended discussion of examples of such systems)
certain difficulties appear in generalizing this approach since the construction, as described in
\cite{FoiasTemam1975, BCFM95, FoiasManleyRosaTemam1}, seems to rely in an essential way on the existence
of a compact \emph{absorbing} set.

More recent works, \cite{Wang2009, LukaszewiczRealRobinson2011}, have been able to remove this
compactness assumption by restricting $X$ to be a reflexive Banach space and by requiring an
additional `weak-to-weak' continuity assumption on $\{ S(t)\}_{t \geq 0}$.  In essence, these
works make use of the fact that bounded sets are \emph{weakly compact}
so that the arguments employed
rely on the structure of reflexive Banach spaces and their associated weak topologies.
While these requirements may be met in certain examples, the theory necessitates
the verification of this additional weak-to-weak assumption which can be
involved in practice (see e.g. \cite{Rosa98}).  Moreover, there are other interesting classes of
dissipative dynamical systems which do not fall into this category. 
We develop some examples in detail below in this connection.

We show here that these cumbersome assumptions,
imposed in \cite{Wang2009, LukaszewiczRealRobinson2011},
are in fact totally unnecessary.  We demonstrate that the corresponding results
in \cite{FoiasManleyRosaTemam1} extend to any continuous semigroup $\{ S(t)\}_{t \geq 0}$, evolving
on any complete, separable metric space, which possesses a global attractor.   Moreover, our method of proof is
more elementary in character;  we extend the methods
of \cite{FoiasManleyRosaTemam1} via a simple topological observation which limits the growth
of continuous functions in a neighborhood of any compact subset of a metric space.  
See Lemma~\ref{thm:BasicMetricSpaceBS}
below.

With this background in mind we may describe the main abstract results in
this work as follows:  Suppose that $(X,d)$ is any complete,
separable metric space which is acted on by a continuous semigroup
$\{S(t)\}_{t \geq 0}$ possessing a global attractor $\mathcal{A}$.
We show that for any probability distribution $\brlMsr_0$  of
initial conditions over $X$, that there is an invariant
probability measure $\brlMsr$ for $\{S(t)\}_{t \geq 0}$, which has its support
contained in the global attractor $\mathcal{A}$, such that
    \begin{equation}\label{eq:WeakErgodicAveGeneral}
       \int_{X} \varphi(x) \dr\brlMsr(x)
       = \LIM \frac{1}{T}\int_0^T \int_X \varphi(S(t) x) \dr\brlMsr_0(x) \dr t,
    \end{equation}
    for all real valued continuous bounded observables $\varphi$.
  Note that, as in \cite{LukaszewiczRealRobinson2011}, we show here that \eqref{eq:WeakErgodicAveGeneral}  holds
for an arbitrary probability distribution of initial conditions and not just for  individual trajectories emanating from some $x_0 \in X$; but again we
do so without any restrictions on  $X$ beyond that it be a complete and separable metric space.
On the other hand, in the case
    when $\brlMsr_0$ is a Dirac measure supported on some point $x_0 \in X$ (i.e.  when $\brlMsr_0 = \delta_{x_0}$) we show
    that \eqref{eq:WeakErgodicAveGeneral} holds for  \emph{any} continuous real valued observable $\varphi$.
In any case, as with the previous results in this direction,  \eqref{eq:WeakErgodicAveGeneral} may
be seen as establishing a kind of weak notion of ergodicity; by replacing the operation $\lim_{T\rightarrow \infty}$ with the operation
 $\LIM$ we show that the equality of time averages and ensemble averages can be obtained.
Note finally that our results generalize naturally to a non-autonomous or a stochastic
setting but this comes at the cost of significant additional technicalities.
We refer the interested reader to \cite{ChekrounGlattHoltz2011a,ChekrounGlattHoltz2011b}
where these generalizations will be carried out in the non-autonomous and stochastic
cases respectively.

The initial motivations that led us to discover the results
appearing herein arose in ongoing work on nonlinear
partial differential equations of parabolic type with memory effects
added in the diffusion terms (see
\cite{ChekrounDiPlinioGlattHoltzPata2010,
ChekrounGlattHoltz2011a,ChekrounGlattHoltz2011b}).  Such systems
have a `hyperbolic character' in comparison to their more classical cousins;
they provide an interesting class of examples which
do not possess in general a compact absorbing set for the associated semigroup. Of course there are many other important evolution equations
which exhibit non-compact semigroups; for example, the  Navier-Stokes equations on unbounded (i.e. domains where the Poincar\'e
inequality holds) \cite{Rosa98}, retarded equations with
infinite delay \cite{HaleLunel93}, neutral functional differential
equations \cite{HaleLunel93},  certain partial functional differential
equations \cite{Wu96}, the linearly damped nonlinear wave equation
as well as other equations of (partially) hyperbolic type
\cite{Ladyzhenskaya91, Temam3}.  Note that restricting consideration
to semigroups evolving on reflexive Banach spaces
is too  stringent for many of the above cited equations.
For instance in the case of neutral
functional differential equations, cf. \cite{HaleLunel93},
the phase-space is typically
$C([-\tau,0],\mathbb{R}^n)$ for
some delay time $\tau > 0$.

Here, in order to illustrate the flexibility of our main abstract
results, we will study two dynamical systems in detail. In each case
we consider systems
with memory which generate non-compact semigroups. We first consider
a model for a viscoelastic fluids similar to the 2D Navier-Stokes
equations but incorporating non-local, integro-differential
diffusive terms that depend on the past history of the flow.  See
e.g. \cite{Joseph90, AgranovichSobolevskii98, Orlov99,
GattiGiorgiPata} and below for further mathematical and physical
background. We then turn to a class of neutral delay differential
equations (NDDEs).  These are functional differential
equations with dependence on the past values of the solution and its
time derivatives.   Such systems arose in relation to the study of
certain line transmission problems modeled by the telegrapher's
equation  \cite{Brayton_Miranker64}, and have been since
encountered in various engineering and physical applications
involving other hyperbolic PDEs;  see e.g. \cite{KH10} and Remark~\ref{rmk:GFDapp} 
below for connections to systems in geophysical fluid dynamics.
For further
background on the mathematical theory of NDDEs see e.g.
\cite{HaleLunel93}.

The manuscript is organized as follows:  We begin by briefly reviewing
some mathematical generalities and setting notations that will be employed
throughout the rest of the work.  The main results are then given
in precise terms in Theorems~\ref{thm:Prop1} and \ref{thm:Prop2}
below.  Complete, self-contained, proofs of both results are next given in
Section~\ref{sec:ProofOfAbsResults}.  Finally
in Section~\ref{sec:Examples}, we turn to the study of the two
concrete examples, the Navier-Stokes equations with memory and certain classes of NDDEs,
establishing novel results linking invariant measures to temporal averages for these systems in
Theorem~\ref{thm:JeffGlobalAttractor} and Theorems~\ref{Thm_existence_attractor},~\ref{THM_IM_NDDE}
respectively.

\section{Notations, mathematical preliminaries and statement of the main results}
\label{sec:MainResults}

Before stating the main abstract results of the work in precise terms (Theorems~\ref{thm:Prop1}, \ref{thm:Prop2})
we first review some basic definitions and other essential mathematical preliminaries setting notations
that will be used below.
Throughout the rest of the article we will always take
$(X,d)$ to be an arbitrary complete, separable metric space and consider
a continuous semigroup $\{S(t)\}_{t \geq 0}$ on $X$;
more precisely we assume that $S(0)=\mbox{Id}_X$,
$S(t+s)=S(t)S(s)$ for all $t,s\in\mathbb{R}^{+},$
and that $S:\mathbb{R}^{+}\times X \rightarrow X$ is separately continuous.

Recall that a global attractor $\mathcal{A}$ is a compact subset of $X$ that is invariant under $S(t)$,
i.e. such that $S(t) \mathcal{A} = \mathcal{A}$ for all $t \geq 0$ and which attracts all bounded subsets of $X$
viz.
\begin{displaymath}
    \lim_{t \rightarrow \infty} d_H(S(t)B, \mathcal{A})
    = 0, \quad
    \textrm{ for all } B \subset X, B \textrm{ bounded},
\end{displaymath}
where $d_H$ is the Hausdorff semi-distance
\begin{displaymath}
    d_H(E, F) := \sup_{x \in E} \inf_{y \in F} d(x, y)
    \textrm{ for any } E, F \subset X.
\end{displaymath}
The study of attractors is an extensive and well-developed subject see
e.g. \cite{Temam3, Hale1988, Ladyzhenskaya91, Chueshov99, Robinson1, ChepyzhovVishik2002, MiranvilleZelik2008}.
For the abstract results, Theorems~\ref{thm:Prop1}, \ref{thm:Prop2}, we will
assume that $\{S(t)\}_{t \geq 0}$ possesses a global attractor $\mathcal{A}$.
On the other hand, for the concrete
examples considered in Section~\ref{sec:Examples}, we employ the
following useful and rather general sufficient condition for the existence
of a global attractor.
Note also that this result follows immediately
as a special case of the
characterizations of semigroups possessing
a global attractor appearing in e.g. \cite[Theorem 3.8]{Ma_Wang_Zhong}
or \cite[Theorem 11]{ChepyzhovContiPata2011}.
\begin{Prop}\label{thm:AttractorCriteria}
Let $H$ be a Banach Space (with an associated norm $\| \cdot \|$) and consider
$\{S(t)\}_{t \geq 0}$ a continuous semigroup acting on $H$.  Suppose that
\begin{itemize}
\item[(i)] there exists a bounded set $\mathfrak{B} \subset H$ such that
for every  $B \subset H$, $B$ bounded, there exists $t^* = t^*(B) > 0$
such that $S(t)B \subset \mathfrak{B}$ for every $t \geq t^*$.
\item[(ii)] For $t \geq 0$ we may split $S(t)$ as $S(t) = S_{1}(t) + S_{2}(t)$ such that, for every $K>0$,
$$\sup_{x \in H: \|x\| \leq K} \|S_1(t)x\| \xrightarrow{t \rightarrow \infty}  0$$
and for every bounded set $B$ and every $t > 0$, $S_{2}(t) B$ is a precompact subset of $H$.
\end{itemize}
Then $\{S(t)\}_{t \geq 0}$ has a  (connected) global attractor $\mathcal{A}$ which is the
omega limit set of $\mathfrak{B}$ i.e.
$$
    \mathcal{A} = \omega(\mathfrak{B}) := \bigcap_{t \geq 0} \overline{\bigcup_{s \geq t} S(s) \mathfrak{B}}.
$$
\end{Prop}

We next introduce some notations and recall some basic properties associated with
probability measures defined on the general metric space $X$.  Let $Pr(X)$ be the set of all Borel
probability measures on $X$ with $\mathcal{B}(X)$ the associated
collection of Borel measurable sets.  For $\brlMsr \in Pr(X)$, we define
$\mbox{supp}(\brlMsr)$ to be the smallest closed set $E$ such that $\brlMsr(E) = 1$.
See e.g. \cite{Rudin1987}.\footnote{
Note that a related, but purely measure theoretic notion of a `carrier' is also sometimes used in this connection.
However, both as it concerns us here and in its usage in previous related works, cf. \cite{FoiasManleyRosaTemam1,BCFM95}, 
the two notions can be seen to be applied
in a completely equivalent fashion.  We will make this equivalence precise after the statement of the
main Theorems~\ref{thm:Prop1},\ref{thm:Prop2} in Remark~\ref{rmk:PropMisc}, (iii) below.
}
Take $C(X)$ (resp. $C_b(X)$) to be
the collection of real-valued continuous (resp. bounded continuous)
functions defined on $X$. A measure $\brlMsr \in Pr(X)$ is said to be {\it
invariant} (relative to $\{ S(t)\}_{t \geq 0}$)  if
\begin{equation}\label{eq:Invar}
 \brlMsr(E) = \brlMsr( S(t)^{-1} E),
  \textrm{ for all } t \geq 0 \textrm{ and every } E \in \mathcal{B}(X),
\end{equation}
or equivalently if
\begin{equation}\label{eq:InvarWeak}
\int_X \varphi(x) \dr \brlMsr(x) =  \int_X \varphi(S(t)x) \dr \brlMsr (x),
 \textrm{ for all } t \geq 0 \textrm{ and every } \varphi \in C_b(X).
\end{equation}
Note that, for every invariant measure $\brlMsr $, $supp(\brlMsr )$ is
contained in the global attractor $\mathcal{A}$.
For completeness we recall the proof of this
fact, which is elementary, in an Appendix.
See Lemma~\ref{thm:InvarOnAttr} below.

Recall that a sequence $\{\brlMsr _{n}\}_{n \geq 0} \subset Pr(X)$ is said
to converge weakly to a measure $\brlMsr $ iff
$\lim_{n \rightarrow \infty}\int_{X}\varphi d \brlMsr _{n}= \int_{X} \varphi d \brlMsr $,
for every $\varphi \in C_{b}(X)$.  As such $\{\brlMsr _{n}\}_{n \geq 0}$
may be said to be  \emph{weakly compact} if we can extract from  $\{\brlMsr _{n}\}_{n \geq 0}$
a weakly convergent subsequence. On the other hand we say such
a collection $\{\brlMsr _{n}\}_{n \geq 0}$ is \emph{tight} if, for every $\epsilon >0$
there is a corresponding compact set $K_{\epsilon} \subset X$ so that
$\brlMsr _{n}(K_{\epsilon}) \geq 1 - \epsilon$, for every $n$.
Classically these two notions, tightness
and weak compactness,
are equivalent, a result usually referred to as Prokhorov's theorem. See e.g.
 \cite{Billingsley1}.

Let us also recall a special case of the classical Kakutani-Riesz Representation
theorem, as suits for our purposes below.  See e.g. \cite{Rudin1987} for further details.
\begin{Lem}\label{thm:KakRiesz}
   Let $K$ be a compact Hausdorff space and suppose that $\mathfrak{L}$
   is a positive linear functional on $C(K)$ (the continuous real valued functions
   on $K$ with the usual $\mbox{sup}$ norm).
   Then there exists a unique positive Borel measure $\brlMsr $  on $K$ such that,
   for every $\varphi \in C(K)$,
   $
   \mathfrak{L}(\varphi) = \int_{K} \varphi(x) d \brlMsr (x)
   $.
\end{Lem}

Finally we turn to the notion of a {\it generalized Banach limit} which
is defined as follows:
\begin{Def}\label{def:GBL}
Consider the collection $\mathcal{B}_+$ of all bounded real-valued
functions on $[0,\infty)$ endowed with sup norm.
A generalized Banach limit, which we denote by
$\underset{t\rightarrow \infty}{\rm LIM}$,
is any linear functional on $\mathcal{B}_+$ such
that
\begin{itemize}
\item[(a)] $\LIM g(t)\geq 0$ for
all $g\in \mathcal{B}_+$ with $g(s) \geq 0$, for all $s\geq 0$.
\item[(b)] $\LIM g(t) =\underset{t\rightarrow
\infty}\lim g(t)$ for all $g\in\mathcal{B}_+$ for which the usual
limit exists.
\end{itemize}
\end{Def}
\noindent
It is not hard to establish the existence of such
a positive linear functional as a consequence of
the Hahn-Banach theorem.
Note also that, for any such $\LIM$, it may be shown that
\begin{equation}\label{eq:LIMbasicinfsup}
\left| \LIM g(T) \right|
\leq  \limsup_{T \rightarrow \infty} |g(T)|,
\end{equation}
for any $g \in \mathcal{B}_{+}$.
We will use this observation frequently below.
See e.g. \cite{Lax2002} for further background and properties.

With these preliminaries in hand we now state the first main result
which shows that \eqref{eq:WeakErgodicAveGeneral}
holds for any continuous observable in the case when $\brlMsr _0$ is a Dirac measure.
\begin{Thm}\label{thm:Prop1}
    Suppose that $(X,d)$ is a complete, separable metric
    space and $\{ S(t)\}_{t \geq 0}$ is a continuous
    semigroup on $X$ that possesses a global attractor
    $\mathcal{A}$.  Fix a generalized Banach limit $\LIM$.  Then,
    for any $x_0 \in X$, there exists a unique invariant measure
    $\brlMsr  \in Pr(X)$  for $\{ S(t)\}_{t \geq 0}$ such that
    \begin{equation}\label{eq:WeakErgodicAvePt}
       \int_{X} \varphi(x) \dr \brlMsr (x)
       = \LIM \frac{1}{T}\int_0^T \varphi(S(t) x_0)  \dr t,
       \quad \textrm{ for any } \varphi \in C(X),
    \end{equation}
    and such that $\mbox{supp}(\brlMsr ) \subseteq \mathcal{A}$.
\end{Thm}
\noindent The second result establishes \eqref{eq:WeakErgodicAveGeneral}
for any initial probability measure $\brlMsr _0$ but requires further in this generality
that the observable be both continuous and bounded.
\begin{Thm}\label{thm:Prop2}
    Suppose that $(X,d)$ is a complete, separable metric
    space and $\{ S(t)\}_{t \geq 0}$ is a continuous
    semigroup on $X$ that possesses a global attractor
    $\mathcal{A}$.  Fix a generalized Banach Limit $\LIM$.
    Then, for any $\brlMsr _0 \in Pr(X)$, there exists a unique invariant
    measure $\brlMsr  \in Pr(X)$ for $\{ S(t)\}_{t \geq 0}$ such that
    \begin{equation}\label{eq:WeakErgodicAve}
       \int_{X} \varphi(x) \dr \brlMsr (x)
       = \LIM \frac{1}{T}\int_0^T \int_X \varphi(S(t) x) \dr \brlMsr _0(x) \dr t,
       \quad \textrm{ for any } \varphi \in C_b(X),
    \end{equation}
    and such that $\mbox{supp}(\brlMsr ) \subseteq \mathcal{A}$.
\end{Thm}
 We conclude this section with some remarks concerning
Theorems~\ref{thm:Prop1}, \ref{thm:Prop2}.
\begin{Rmk}\label{rmk:PropMisc}
\mbox{}
\begin{itemize}
\item[(i)] The uniqueness of the invariant measures $\brlMsr$
satisfying \eqref{eq:WeakErgodicAvePt} or \eqref{eq:WeakErgodicAve}
follows as a direct consequence of the regularity of borel probability measures on 
metric spaces;
see e.g. \cite{Billingsley1}.  We therefore need only to establish the \emph{existence}
of such measures in the proofs below.
\item[(ii)] By considering, for $x_{0} \in X$, the Dirac measure $\brlMsr _{0} = \delta_{x_0}$ we
partially recover Proposition~\ref{thm:Prop1} from
Proposition~\ref{thm:Prop2}. Notwithstanding, we separate the two
results since, in the former case of Proposition~\ref{thm:Prop1}, we
are able to establish \eqref{eq:WeakErgodicAvePt} in the larger class
of test functions $C(X)$.
Note that the use of the smaller collection $C_b(X)$ for
the space of test functions in Proposition~\ref{thm:Prop2} allows us
in particular to work in the topology of weak convergence of measures.
This is needed to be able to pass to a limit in a sequence of approximating measures
via the Prokhorov Theorem.   See Step 2 of the proof of Proposition~\ref{thm:Prop2}
below.
\item[(iii)] Theorem~\ref{thm:Prop1}, in the given generality of complete separable metric spaces
  acted on by arbitrary semigroups possessing a global attractor, also appears in \cite{LukaszewiczRealRobinson2011}.
  We give a different proof of Theorem~\ref{thm:Prop1} below, based on Lemma~\ref{thm:BasicMetricSpaceBS}.   
  Lemma~\ref{thm:BasicMetricSpaceBS} is also used in an essential way in the proof of Theorem~\ref{thm:Prop2}.
  In contrast to Theorem~\ref{thm:Prop1}, Theorem~\ref{thm:Prop2} establishes results analogous to those appearing 
  in previous works in a much greater generality and is thus new.
\item[(iv)]
In other related works, e.g. \cite{FoiasManleyRosaTemam1,BCFM95},
the notion of a \emph{carrier} is sometimes used as an alternative
to the support of a measure. Recall that an element $\brlMsr \in
Pr(X)$ is said to be \emph{carried by} a set $E \in \mathcal{B}(X)$
(or that $E$ is a \emph{carrier for} $\brlMsr$) if $\brlMsr(E) = 1$.
Of course, $\brlMsr$ is carried by $\mbox{supp}(\brlMsr)$ (the
smallest closed set of full measure), but this particular carrier is
not unique in general (and rarely would be in practice). In any
case,  since the global attractor $\mathcal{A}$ is closed, when say
in Theorems~\ref{thm:Prop1}, \ref{thm:Prop2} that
`$\mbox{supp}(\brlMsr) \subseteq \mathcal{A}$' we may equivalently
state that `$\brlMsr$ is carried by $\mathcal{A}$'.
\item[(v)] The so called Krylov--Bogoliubov procedure provides another means of associating
invariant measures with time averages starting from a given fixed initial measure $\mu_0$. 
In \cite{LukaszewiczRealRobinson2011} it was shown that this procedure could be used to 
establish the existence of invariant measures for the class of dynamical systems considered
in, for example, Theorems~\ref{thm:Prop1}, \ref{thm:Prop2}.
\end{itemize}
\end{Rmk}

\section{Proof of abstract results}
\label{sec:ProofOfAbsResults}
We turn in this section to the proof of Theorem~\ref{thm:Prop1}, \ref{thm:Prop2}.
Both proofs  lean heavily on the following
general topological lemma which is
established along elementary lines.
 \begin{Lem}\label{thm:BasicMetricSpaceBS}
     Let $(X,d)$ be a metric space and consider any
     compact subset $K \subset X$.
     For $\epsilon > 0$ we let,
     $K_\epsilon := \{ x \in X: \inf_{y \in K} d(x,y)  < \epsilon \}.$
Then the following properties hold:
   \begin{itemize}
     \item[(a)]       For every $\varphi \in C(X)$,
       there exists $\epsilon >0$ so that
       $\sup\limits_{x \in K_\epsilon} |\varphi(x)| < \infty. $
   \item[(b)] Suppose $\varphi, \psi \in C(X)$ are
     such that $\varphi(x) = \psi(x)$ for every $x \in K$, then
     for every $\epsilon > 0$ there exists $\delta > 0$ such
     that,
     $ \sup\limits_{ x \in K_\delta} |\varphi(x) - \psi(x)| < \epsilon$.
   \end{itemize}
\end{Lem}
\begin{proof}
For the first item, (a), fix $\varphi \in C(X)$. For every $x \in K$
we may
  choose $\delta = \delta_x > 0$ so that, for every
  $ y \in B(x, \delta_x) := \{y : d(x, y) < \delta_x\}$, $|\varphi(x) -
   \varphi(y)| < 1$.  Choosing numbers $\delta_{x} > 0$ in this manner
   we may form the open cover $\mathcal{C} = \{ B(x, \delta_x/3) : x \in K\}$
  for $K$.  Since $K$ is compact, we can extract from this cover a
  finite sub-cover
  $ \mathcal{C}' = \{ B(x_1, \delta_{x_1}/3), \ldots , B(x_n,
   \delta_{x_n}/3) \}$.   Take
   $\epsilon = \frac{\min\{ \delta_{x_1}, \ldots, \delta_{x_n}\}}{3}$
   and let $M = 1 + \max_{j = 1, \ldots, n}
   |\varphi(x_j)|$.

   Given any
   $x \in K_\epsilon$
   we may choose $y \in K$ such that $d(x,y) < 2 \epsilon$.  Since
   $\mathcal{C}'$ covers $K$ we may pick $x_j$ such that $d(y, x_j)
   < \delta_{x_j}/3$.  Combining these two observations we conclude
   perforce that $d(x, x_j) < 2 \epsilon + \delta_{x_j}/3 \leq \delta_{x_j}$ so that $|\varphi(x)| \leq M$.
   Since $x \in K_\epsilon$ was arbitrary to begin with this gives
   (a).

  We turn to the second item (b).  Fix $\epsilon > 0$.
  For every $x \in K$ we choose $\gamma_x > 0$ so that
  $ |\varphi(x) - \varphi(y)| + |\psi(x) - \psi(y)| < \epsilon $
  whenever $y \in B(x, \gamma_x)$.  Again, due to the
  compactness of $K$, we may cover $K$ with
  a finite collection $\{ B(x_1, \gamma_{x_1}/3) , \ldots, B(x_n, \gamma_{x_n}/3) \}$.
  Take $\delta = \frac{\min\{ \gamma_{x_1}, \dots, \gamma_{x_n} \}}{3}$.  Similarly to the previous case
  we observe
  that $K_\delta \subset \cup_{k =1}^n B(x_k, \gamma_k)$.
Fix arbitrary $y \in K_\delta$ and choose $k$ so that $y \in B(x_k, \gamma_k)$.
  Noting that
  $\varphi(x_k) = \psi(x_k)$, this implies
      $ |\varphi(y) - \psi(y)| = |\varphi(y) - \varphi(x_i) + \psi(x_i) - \psi(y)|
      \leq |\varphi(y) - \varphi(x_i)| + |\psi(x_i) - \psi(y)|
      < \epsilon$, as needed for (b).
       The proof is therefore complete.
\end{proof}

\subsection*{Proof of Theorem~\ref{thm:Prop1}:}
Fix any $x_{0} \in X$.  We proceed in steps.  First we show that the
generalized  Banach limit on the right hand side of
\eqref{eq:WeakErgodicAvePt} exists.  We then show that
the resulting functional is uniquely defined by the restriction of $\varphi$
to the global attractor $\mathcal{A}$.  Since $\mathcal{A}$ is compact,
this is sufficient to infer the needed
$\brlMsr$ via an application of the Kakutani-Riesz theorem. In the final
step we show that the $\brlMsr $ that we have found is indeed invariant.

  \noindent {\bf Step 1: Existence of the positive linear functional}.   
  To carry out this first step, we show that the map defined by
  $T \mapsto \tfrac{1}{T} \int_0^T \varphi(S(t) x_0) \mbox{d}t$
  is bounded over $[0,\infty)$ so that
  \begin{equation}\label{eq:LinFnDefProp1}
  \mathfrak{L}_{x_0}(\varphi) := \LIM \frac{1}{T} \int_0^T \varphi(S(t)x_0) \mbox{d}t,
  \quad \textrm{ for } \varphi \in C(X),
  \end{equation}
  is well defined as a positive linear functional on $C(X)$ (cf. Definition~\ref{def:GBL}).
Since, by assumption  $\{ S(t) \}_{t \geq 0}$ possesses a global attractor $\mathcal{A}$, we have
that,
  \begin{equation}\label{eq:absorbingAttractorNb}
  \textrm{for every } \epsilon >0 \textrm{, there exists a time }T_\epsilon \geq 0 \textrm{ such that }
    S(t)x_0 \in \mathcal{A}_\epsilon
       \textrm{ for every }
    t \geq T_\epsilon,
  \end{equation}
  where here and below, $\mathcal{A}_\epsilon = \{ x \in X :  \inf_{y \in \mathcal{A}} d(x,y) < \epsilon\}$.  Applying
  Lemma~\ref{thm:BasicMetricSpaceBS}, (a) to $\mathcal{A}$ and $\varphi$
  we may choose $\epsilon > 0$ so that
    $K_A := \sup_{x \in \mathcal{A}_\epsilon} |\varphi(x)| < \infty$.
 Taking $T_\epsilon$ as required in \eqref{eq:absorbingAttractorNb} for this value of
 $\epsilon$ we find also that
   $K_I := \sup_{t \in [0,T_\epsilon]} |\varphi(S(t) x_0)| < \infty$
 owing to the fact that the interval $[0,T_\epsilon]$ is compact and that $t \mapsto S(t)$ is
 continuous.  By now taking $K = \max\{K_I, K_A\}$, we infer,
 for every $T > 0$, that
   $\left|\frac{1}{T} \int_0^T \varphi( S(t) x_0) \dr t \right| \leq K < \infty$
 as needed to finally justify the definition of $\mathfrak{L}_{x_0}$ given in
 \eqref{eq:LinFnDefProp1}.

\noindent {\bf Step 2: Restriction to $\mathcal{A}$ and the
application of the Kakutani-Riesz theorem.} The next step will be to
show that $\mathfrak{L}_{x_0}(\varphi)$ depends only on the values of $\varphi$ on $\mathcal{A}$.  More precisely,
we establish that if $\varphi(x) = \tilde{\varphi}(x)$
 for every $x \in \mathcal{A}$ then $\mathfrak{L}_{x_0}(\varphi) = \mathfrak{L}_{x_0}(\tilde{\varphi})$.
 To this end we
 fix any $\epsilon >0$ and according to Lemma~\ref{thm:BasicMetricSpaceBS}, (b)
 choose a corresponding $\delta >0$ such that
  $\sup\limits_{ x \in \mathcal{A}_\delta} |\varphi(x) - \tilde{\varphi}(x)| < \epsilon$.
We now take $T_\delta$
so that $S(t)x_{0} \in \mathcal{A}_{\delta}$ for each $t \geq T_{\delta}$
 and let
   $\tilde{K}_\delta :=
   \sup_{t \in [0,T_\delta]} (|\varphi(S(t) x_0)| + |\tilde{\varphi}(S(t) x_0)|)$.
For the reasons noted for $K_I$ in the previous step above, $\tilde{K}_{\delta}$ is finite.
Using the basic properties
 of $\LIM$, \eqref{eq:LIMbasicinfsup},
 we estimate
 \begin{displaymath}
   \begin{split}
     \left|\mathfrak{L}_{x_0}(\varphi - \tilde{\varphi}) \right| =&
     \left| \underset{t \rightarrow \infty}{LIM}
    \frac{1}{T} \int_0^T \varphi(S(t)x_0) -\tilde{\varphi}(S(t)x_0)\dr t
    \right|
      \leq     \limsup_{T \rightarrow \infty} \left|
   \frac{1}{T} \int_0^T \varphi(S(t)x_0) -\tilde{\varphi}(S(t)x_0)\dr t
   \right|\\
     \leq&     \limsup_{T \rightarrow \infty}
  \frac{1}{T} \int_0^{T_\delta} |\varphi(S(t)x_0) -\tilde{\varphi}(S(t)x_0)|\dr t
   +    \limsup_{T \rightarrow \infty}
  \frac{1}{T} \int_{T_\delta}^T |\varphi(S(t)x_0) -\tilde{\varphi}(S(t)x_0)|\dr t\\
  \leq&     \limsup_{T \rightarrow \infty} \frac{T_\delta \tilde{K}_\delta}{T}
              + \limsup_{T \rightarrow \infty} \frac{1}{T} \int_{T_\delta}^T |\varphi(S(t)x_0) -\tilde{\varphi}(S(t)x_0)|\dr t\\
  \leq&    \limsup_{T \rightarrow \infty} \frac{T_\delta \tilde{K}_\delta}{T} +
  \limsup_{T \rightarrow \infty} \frac{(T- T_\delta)\epsilon }{T} \leq \epsilon.
  \end{split}
\end{displaymath}
 Since, to begin with, the choice of $\epsilon >0$ was arbitrary we conclude perforce that
 $\mathfrak{L}_{x_0}(\varphi - \tilde{\varphi}) = 0$ as desired.

 With this in hand we may now unambiguously define
 \begin{equation}\label{eq:GDefPt}
 \mathfrak{G}(\psi) = \mathfrak{L}_{x_{0}}( \ell(\psi)), \quad
 \textrm{ for } \psi \in C(\mathcal{A}),
 \end{equation}
where we take $\ell: C(\mathcal{A}) \rightarrow C(X)$ to be an extension
operator, $\ell(\psi)(x)
= \psi(x)$ for $x \in \mathcal{A}$, such that
\begin{equation}\label{eq:GoodExtDugndji}
    \inf_{x \in X} \ell(\varphi)(x) =  \inf_{x \in \mathcal{A}} \varphi(x),
    \quad   \sup_{x \in X} \ell(\varphi)(x) =  \sup_{x \in \mathcal{A}} \varphi(x).
\end{equation}
The existence of such an extension operator is
guaranteed by the Dugundji extension theorem, \cite[Theorem 4.1]{Dugundji1951}.
One may readily verify that $\mathfrak{G}$ is linear.  With
\eqref{eq:GoodExtDugndji}, \eqref{eq:LinFnDefProp1}
and Definition~\ref{def:GBL}, (a)
we see that $\mathfrak{G}$ is also positive.
Thus, by the Kakutani-Riesz representation theorem, recalled
above in Lemma~\ref{thm:KakRiesz}, there exists a unique positive,
finite, Borel measure $\brlMsr $ on $\mathcal{A}$ such that
\begin{equation}\label{eq;ReizeDef}
 \mathfrak{G}(\psi) = \int_{\mathcal{A}} \psi(x) \dr \brlMsr (x),  \quad
 \textrm{ for } \psi \in
 C(\mathcal{A}).
\end{equation}
Abusing notation slightly, we extend $\brlMsr $ to a Borel measure on all of $X$ by taking
$\brlMsr (E) := \brlMsr (E \cap \mathcal{A})$, $E \in \mathcal{B}(X)$.
Clearly, $\brlMsr (\mathcal{A}^c) =0$, and so with
\eqref{eq:LinFnDefProp1}, \eqref{eq:GDefPt}, \eqref{eq;ReizeDef}
it follows that, for every $\varphi \in C(X)$,
  $\mathfrak{L}_{x_{0}}(\varphi) = \int_{\mathcal{A}} \varphi(x) \dr \brlMsr (x) = \int_{X}
  \varphi(x) \dr \brlMsr (x)$, which is \eqref{eq:WeakErgodicAvePt}.
To see that $\brlMsr $ lies
in $Pr(X)$ we simply insert $\varphi \equiv 1$ into \eqref{eq:WeakErgodicAvePt}.

\noindent {\bf Step 3.  The invariance of $\brlMsr $.}  For the final step we show
that the $\brlMsr $ we have found in Step 2 is indeed invariant by
establishing \eqref{eq:InvarWeak}.  To this end fix any $t^{*} \geq 0$
and any $\varphi \in C_b(X)$.  We compute:
\begin{equation}\label{eq:proofMuInvariance}
  \begin{split}
    \int_{X} \varphi(S(t^{*}) x) \dr \brlMsr (x) &=
    \LIM \frac{1}{T} \int_{0}^{T}  \varphi(S(t^{*})S(t)x_{0}) \dr t =
    \LIM \frac{1}{T} \int_{0}^{T}  \varphi(S(t+t^{*})x)  \dr t\\
    &= \LIM \frac{1}{T} \int_{t^{*}}^{T+t^{*}} \varphi(S(t)x_{0})  \dr t\\
    &= \int_{X} \varphi( x) \dr \brlMsr (x) +
    \LIM \frac{1}{T} \int_{T}^{T+t^{*}} \varphi(S(t)x_{0}) \dr t -
    \LIM \frac{1}{T} \int_{0}^{t^{*}} \varphi(S(t)x_{0})  \dr t\\
    &= \int_{X} \varphi( x) \dr \brlMsr (x).
  \end{split}
\end{equation}
Note that the first equality is justified since $\psi(\cdot) = \varphi(S(t^{*})\cdot)
\in C_b(X)$.
For the final equality above we note that both
$\LIM \frac{1}{T} \int_{0}^{t^{*}} \varphi(S(t)x_{0})  d t$ and $\LIM \frac{1}{T} \int_{T}^{T+t^{*}} \varphi(S(t)x_{0}) \dr t$
are zero due to the boundedness of $\varphi$ and the basic properties of $\LIM$, Definition~\ref{def:GBL}, (b).
We have thus completed the proof of Theorem~\ref{thm:Prop1}.

\begin{Rmk}
    For the proof of Theorem~\ref{thm:Prop1},
    the fact that $\mbox{supp}(\brlMsr ) \subseteq \mathcal{A}$ follows directly from
    the application of the Kakutani-Riesz theorem
    on the global attractor $\mathcal{A}$.  The situation is
    different in e.g. \cite{FoiasManleyRosaTemam1} where Kakutani-Riesz theorem
    is applied on a compact \emph{absorbing set}.  Here one needs to rely on the
    general result, recalled in Lemma~\ref{thm:InvarOnAttr}
    and cf. \cite[Chapter 4, Theorem 4.1]{FoiasManleyRosaTemam1},
    that the support of an
    invariant measure is always contained in the global attractor $\mathcal{A}$.
    On the other hand we use Lemma~\ref{thm:InvarOnAttr} in Step 2 of the proof of
    Theorem~\ref{thm:Prop2} for the general case of an
    $\brlMsr _0 \in Pr(X)$ with an unbounded support.
\end{Rmk}

\subsection*{Proof of Theorem~\ref{thm:Prop2}:}
Let $\brlMsr _0 \in Pr(X)$ be given.  Since we are assuming here that $\varphi$ is bounded,
in contrast to the proof of Proposition~\ref{thm:Prop1} it is trivial to define:
\begin{equation}\label{eq:BLFnDef}
  \mathfrak{L}_{\brlMsr _{0}}(\varphi) = \LIM \frac{1}{T}\int_0^T \int_X \varphi(S(t) x) \dr \brlMsr _0(x) \dr t,
  \quad \textrm{ for }\varphi \in C_b(X).
\end{equation}
As in \cite{LukaszewiczRealRobinson2011}, we  now proceed to establish \eqref{eq:WeakErgodicAve} in two steps.
Initially we assume that $\brlMsr _0$ has a bounded support so that
this support is attracted to $\mathcal{A}$.  At the second step we
drop this assumption and pass to the general case of any $\brlMsr _0
\in Pr(X)$ with compactness arguments involving the Prokhorov
theorem and a suitable sequence of probability measures
$\brlMsr_0^n$ approximating $\brlMsr_0$.

\noindent {\bf Step 1: $\brlMsr _0$ with bounded support.} 
Assuming that $\rm{supp} (\brlMsr _0)$ is a bounded subset of $X$, we
fix any $\varphi, \tilde{\varphi} \in C_b(X)$ with  $\varphi(x) = \tilde{\varphi}(x)$ for all
$x \in \mathcal{A}$.  As in Step 2  of the proof of Proposition~\ref{thm:Prop1}
we would like to show
that $\mathfrak{L}_{\brlMsr _{0}}(\varphi) = \mathfrak{L}_{\brlMsr _{0}}(\tilde{\varphi})$
so that we may once again apply the Kakutani-Riesz theorem.
For this purpose fix $\epsilon > 0$.    Invoking, as above, Lemma~\ref{thm:BasicMetricSpaceBS},
(b) we choose a corresponding $\delta >0$ such that $\sup_{x \in \mathcal{A}_\delta}
|\varphi(x) - \tilde{\varphi}(x)| < \epsilon$.  Since $\mathcal{A}$ is attracting and $\mbox{supp}(\mathfrak{m}_0)$
is assumed to be bounded,
we may choose then choose a time $t^* > 0$ such that $S(t)(\rm{supp}(\brlMsr _0))\subset \mathcal{A}_\delta$
for all $t \geq t^*$.  Using basic properties of $\LIM$, \eqref{eq:LIMbasicinfsup}, we estimate
\begin{displaymath}
\begin{split}
  |\mathfrak{L}_{\brlMsr _{0}}(\varphi - \tilde{\varphi})| \leq&
         \limsup_{T \rightarrow \infty} \frac{1}{T}\int_0^{t^*} \int_X
         \left|\varphi(S(t) x) -  \tilde{\varphi}(S(t) x) \right| \dr \brlMsr _0(x) \dr t \\
     &+ \limsup_{T \rightarrow \infty} \frac{1}{T}\int_{t^*}^T \int\limits_{\rm{supp}(\brlMsr _0)}
     \left| \varphi(S(t) x) -  \tilde{\varphi}(S(t) x)  \right| \dr \brlMsr _0(x) \dr t \\
     \leq& \limsup_{T \rightarrow \infty} \frac{t^* \sup_{x \in X}( |\varphi(x)| + |\tilde{\varphi}(x)|)}{T} +
      \limsup_{T \rightarrow \infty} \frac{(T - t^*) \epsilon}{T} = \epsilon.
\end{split}
\end{displaymath}
Since $\epsilon$ was arbitrary to begin with we infer that
$\mathfrak{L}_{\brlMsr _{0}}(\varphi) =  \mathfrak{L}_{\brlMsr _{0}}(\tilde{\varphi})$
 for all $\varphi, \tilde{\varphi} \in C_b(X)$  with  $\varphi_{| \mathcal{A}} =\tilde{\varphi}_{| \mathcal{A}}$.

With this in hand we are able to define a functional $\mathfrak{G}$ as in \eqref{eq:GDefPt}
and obtain an associated $\brlMsr $
with $\brlMsr (\mathcal{A}^c) = 0$.\footnote{Note that, due to \eqref{eq:GoodExtDugndji}, the extension operator $\ell$ is chosen
to map $C(\mathcal{A})$ into $C_{b}(X)$.}  With this measure $\brlMsr $,
\eqref{eq:WeakErgodicAve} now follows exactly as in Step 1 of the proof of Proposition~\ref{thm:Prop1}.
To show that $\brlMsr  \in Pr(X)$, we take $\varphi \equiv 1$ in
\eqref{eq:WeakErgodicAve} and use that $\brlMsr _0 \in Pr(X)$.

\noindent {\bf Step 2: passage to the general case via the Prokhorov theorem.}
Now suppose that $\brlMsr _0 \in Pr(X)$ is arbitrary.   Since $\brlMsr _0$
is a Borel probability measure it is \emph{tight}; we may choose a
sequence of compact sets $K_n \subset X$ such that
 $ \brlMsr _0( K_n) \geq 1 - 1/n$ (see e.g. \cite{Billingsley1}).
Accordingly we now define a sequence of measures $\brlMsr _0^n \in Pr(X)$ via
\begin{displaymath}
    \brlMsr _0^n (E) = \frac{\brlMsr _0(E \cap K_n)}{\brlMsr _0(K_n)}.
\end{displaymath}
Clearly each $\brlMsr _0^n$ has a bounded support.  Invoking Step 1, we obtain, for each $n$,
a corresponding invariant
measure $\brlMsr ^n$ such that $\brlMsr _0^n$ and $\brlMsr ^n$ satisfy \eqref{eq:WeakErgodicAve}.
Note also that, for each $n$
$\mbox{supp} (\brlMsr ^n) \subseteq \mathcal{A}$.
Hence, since $\mathcal{A}$ is by definition a compact set,
by the Prokhorov theorem (see e.g. \cite{Billingsley1}),
we infer that $\{\brlMsr ^n\}_{n \geq 1}$ must have a weakly convergent subsequence.
Thinning  this sequence of $\brlMsr ^{n}$'s as necessary we find that, for each $\varphi \in C_b(X)$,
\begin{displaymath}
\begin{split}
       \int_{X} \varphi(x) \dr\brlMsr (x)  =&  \lim_{n \rightarrow \infty} \int_{X}  \varphi(x) \dr\brlMsr ^n(x)
       =  \lim_{n \rightarrow \infty} \frac{1}{\brlMsr _0(K_n)}\LIM \frac{1}{T}\int_0^T \int_{K_n} \varphi(S(t) x) \dr\brlMsr _0(x) \dr t\\
       =&  \lim_{n \rightarrow \infty} \frac{1}{\brlMsr _0(K_n)} \left(\LIM \frac{1}{T}\int_0^T \int_X \varphi(S(t) x) \dr\brlMsr _0(x) \dr t-
      \LIM \frac{1}{T}\int_0^T \int_{X \setminus K_n} \varphi(S(t) x) \dr\brlMsr _0(x) \dr t \right)\\
        =&  \LIM \frac{1}{T}\int_0^T \int_X \varphi(S(t) x) \dr\brlMsr _0(x) \dr t,\\
\end{split}
\end{displaymath}
which is \eqref{eq:WeakErgodicAve}.  Note that we may justify the last equality above by observing that
\begin{displaymath}
  \left| \frac{1}{T}\int_0^T \int_{X \setminus K_n} \varphi(S(t) x) \dr\brlMsr _0(x) \dr t  \right|
  \leq \sup_{x \in X} |\varphi(x)| \brlMsr _0(X \setminus K_n)  \leq \frac{\sup_{x \in X} |\varphi(x)|}{n}.
\end{displaymath}

We show that $\brlMsr  \in Pr(X)$ exactly as in the previous step by inserting $\varphi \equiv 1$
into \eqref{eq:WeakErgodicAve}.
The invariance of $\brlMsr $ (for the case of a  $\brlMsr _{0}$ with bounded
support or otherwise) follows from a computation like
\eqref{eq:proofMuInvariance}.
Having
shown that $\brlMsr $ is an invariant probability measure it follows that $\mbox{supp}(\brlMsr ) \subseteq \mathcal{A}$.
Indeed this is a general property of invariant measures that we restate and prove
for the sake of completeness in Lemma~\ref{thm:InvarOnAttr} below.
The proof of Theorem~\ref{thm:Prop2} is now complete.

\section{Application to evolution equations with memory}
\label{sec:Examples}

In this section we consider two concrete dynamical systems which are non-compact but which nevertheless
possess a global attractor.
The common theme between  of these examples is the presence of
\emph{memory terms} in the governing equations; we suppose that the evolution of the state
variables depends on both past and current states the of system.
As noted above in the introduction, each of these examples are not covered under the
previous results appearing in \cite{Wang2009}, \cite{LukaszewiczRealRobinson2011}.
In the first section we study a variation on the Navier-Stokes equations, often
referred to as the Jeffery's Model, which has diffusive memory terms.
Further on in Section~\ref{sec:NDDEs}, we consider a class
of neutral differential equations, i.e. differential equations involving time derivatives
of the state variable at lagged times.
Here, in addition to the fact that the semigroup
is non-compact, the underlying phase space $X$ is a non-reflexive Banach space.

\subsection{A viscoelastic fluids model:  The Navier-Stokes equations with memory}
\label{sec:NSEwMem}

For the first example we consider a variation on the Navier-Stokes
equations, the so called \emph{Jeffery's model}, which incorporates
the past history of the flow through additional diffusive,
integro-differential `memory' terms. See \eqref{eq:JeffMoment} --
\eqref{eq:Jeff1InitData} below for the precise formulation of this
model. Such equations arise in the study of certain viscoelastic
fluids; for example to model dilute solutions of polymers or bubbly
liquids.  See \cite{Joseph90, GattiGiorgiPata}.

Physically speaking, viscoelastic materials exhibit effects of both
elasticity and viscosity. For such materials, the stress is
typically a functional of the past history of the strain, instead of
being a function of the present strain value (elastic) or of the
present value of the time derivative of the strain (viscous). When
an integral term is used for the history dependence, the dynamical
equations become partial integrodifferential equations with a
character somewhere between hyperbolic (elastic) and parabolic
(viscous); see e.g. \cite{FrancfortSuquet1986,RHN87}. Thus, it is
not surprising that the Jeffery's model has interesting hyperbolic
properties in comparison to the classical Navier-Stokes
equations, as we will see below.

Numerous models coming from mathematical physics or mathematical
ecology incorporate such memory terms. For example one may add
memory to the damped or strongly damped wave equation for the study
of wave propagation in materials with memory, or to nonlinear
parabolic equations
 which result in ``reaction-diffusion systems'' that take non-Fickian diffusive effects into account.
 Such parabolic equations are used in the modeling of heat propagation in certain materials \cite{ColemanGurtin1, GurtinPipkin1}
or more recently have been advocated in the study of population dynamics
in situations where the individuals spread according to environmental feedbacks, see e.g. \cite{Pao97,CG06}.

In this  type of  ``reaction-diffusion system'' with non-Fickian diffusive effects,  the
contribution of the medium to the change in
mobile concentration is modeled by a linear density exchange
process which may be viewed as a source term with respect to the memory-less (i.e. Fickian) reaction-diffusion system. This
source term can be expressed as the convolution product of
a time-dependent function  $\kappa$, called the {\it memory kernel}, and some spatial variation of the mobile concentration $u$.
For instance this might lead to a {\it diffusive memory term} of the form $\int_{-\infty}^{t} \kappa(t-s) \Delta u(s)ds$. In any case,
 main feature of this
formulation is that the memory kernel depends only on the
properties of the medium, and is therefore an
intrinsic characteristic of the system. See \cite{GMBDC08} for other types of convolution products
between the time derivative of the mobile concentration and
the memory kernel as arising in non-Fickian dispersion in porous media and other `real-world' situations.

The mathematical study of hyperbolic and parabolic systems modified
by the kind of diffusive memory terms discussed above is now
extensive.  See e.g. \cite{Dafermos1,
GrasselliPata2,ChekrounDiPlinioGlattHoltzPata2010,PataZucchi1,
GattiMiranvillePataZelik2008,GrasselliPata1,GattiGrasselliMiranvillePata,
DiPlinioPata2,DiPlinioPata1,ContiPataSquassina,ChepyzhovPata,
ChepyzhovGattiGrasselliMiranvillePata,GiorgiGrasselliPata99,GrasselliPata2002,GrasselliPata}.
The Jeffery's model we will consider here was previously studied in
\cite{AgranovichSobolevskii98, Orlov99, GattiGiorgiPata}. Since we
are interested in the dynamical properties of this system, we will
treat   \eqref{eq:JeffMoment} - \eqref{eq:Jeff1InitData} as an
autonomous dynamical system within the extended phase space
formalism developed in \cite{Dafermos1} and used subsequently in
many of the works just cited. While the existence of an attractor
for \eqref{eq:JeffMoment} - \eqref{eq:Jeff1InitData} in this
extended phase space setting was studied in \cite{GattiGiorgiPata}
we will establish such results here under a larger class of
conditions on the memory decay kernel $\kappa$ than was possible in
\cite{GattiGiorgiPata}.  We are able to accomplish this feat by
making use of  an appropriate energy functional (see
Lemma~\ref{thm:DispFunctionalProperties} and also
\cite{ChekrounDiPlinioGlattHoltzPata2010,
GattiMiranvillePataZelik2008}). Note that while this more general
condition, \eqref{eq:memkerneldecayHard}, is much more challenging
technically, it allows for the interesting case of a memory term
$\kappa$ exhibiting a linear decay as we describe below in
Remark~\ref{rmk:algDecayForMu}.  Having established that
\eqref{eq:JeffMoment} - \eqref{eq:Jeff1InitData}  possesses a global
attractor within this framework the existence of invariant measures
naturally follow from Theorem \ref{thm:Prop1} and Theorem
\ref{thm:Prop2}. In any case, to the best of our knowledge, no one
has previously considered the existence of invariant measures for
the Jeffery's model, or for that matter any other parabolic or
hyperbolic systems with such diffusive memory terms.

Note that the time asymptotic dynamics of  \eqref{eq:JeffMoment}--\eqref{eq:JeffDivFree} can
be quantified in the cases of a time dependent or a stochastic, white noise type forcing.
This will be carried out elsewhere in \cite{ChekrounGlattHoltz2011a,ChekrounGlattHoltz2011b}.
Of course, we could formally consider a three dimensional version of  \eqref{eq:JeffMoment}--
\eqref{eq:JeffDivFree}. However, as with the classical three dimensional  Navier-Stokes equations,
the existence and uniqueness of strong solutions is unknown. The three dimensional system
is thus far from the reach of the theory we have developed here.

\subsubsection{The Jeffery's model: governing equations and basic assumptions}
\label{sec:NSEwithMemory}

We introduce the model as follows.
Fix a bounded open domain $\DO \subset \R^{2}$ with smooth boundary $\partial \DO$.
On $\DO$ we consider the following integro-differential equation
\begin{align}
  \pd{t} u +  u \cdot \nabla u   &=  \nu \Delta u + \int_{-\infty}^{t} \kappa(t-s) \Delta u(s)ds + \nabla p
                + f, \label{eq:JeffMoment}\\
                \nabla \cdot u &= 0. \label{eq:JeffDivFree}
\end{align}
Here $u = (u_{1},u_{2})$, $p$ represent the flow field and the pressure of a viscous incompressible
fluid filling $\DO$.  The function $\kappa$ determines the dependence on the past history of the flow through the integro-differential `memory' term
$\int_{-\infty}^{t} \kappa(t-s) \Delta u(s)ds$ and
distinguishes \eqref{eq:JeffMoment}--\eqref{eq:JeffDivFree} from the classical Navier-Stokes equations.
As such, $\kappa$ is assumed to
be positive and decreasing.   Further decay conditions for $\kappa$
are be imposed below in \eqref{eq:KappaForm}--\eqref{eq:NEC}.  We will assume throughout
what follows that the external body forcing
$f$ is time independent and that  $f \in L^{2}(\mathcal{D})$.

Of course \eqref{eq:JeffMoment}--\eqref{eq:JeffDivFree}  is supplemented with initial and boundary conditions.
We impose the no-slip (Dirichlet) boundary condition
\begin{equation}\label{eq:JeffDirechelBoundaryCond}
                u|_{\partial\DO} = 0.
\end{equation}
Observe that, in view of the history term $\int_{-\infty}^{t} \kappa(t-s) \Delta u(s)ds$,
the divergence-free vector field $u$ must be known
for all $t\leq 0$ in order to make sense of \eqref{eq:JeffMoment}.
Accordingly, the boundary-value problem
is supplemented with
the `initial condition', or more precisely the \emph{initial past history}:
\begin{equation}                     \label{eq:Jeff1InitData}
                u(t) = u_0(t), \; t \leq 0.
\end{equation}
Of course for each $t \leq 0$, $u_{0}(t)$ respects the conditions \eqref{eq:JeffDivFree}, \eqref{eq:JeffDirechelBoundaryCond}.

Concerning the memory terms in \eqref{eq:JeffMoment}, we assume that $\kappa$
may be written in the form:
\begin{equation}\label{eq:KappaForm}
\kappa(s) := \int_{s}^{\infty} \memK (\sigma)\, d \sigma :=\kappa_0-\int_0^s\memK(\sigma)\, d \sigma,
\end{equation}
for some nonnegative, nonincreasing
function $\memK \in L^{1}(\R^+)$, so that
$\kappa$ is nonincreasing and nonnegative.  We will sometimes refer to this $\memK$ as the {\it memory kernel}  as well, when no confusion is possible with $\kappa$.
Note that $\kappa(0) = \kappa_{0}$ is the `total mass of $\memK$'
i.e.
$\int_0^\infty \memK(s)\, d s=\kappa_0$ and that $\kappa' = -\memK$ for almost every $t \geq 0$.
Throughout the following, we impose the decay condition on $\memK$
\begin{equation}\label{eq:memkerneldecayHard}
  \memK(s + \sigma)  \leq K e^{-\delta \sigma} \memK(s), \; \textrm{ a.e. } s, \sigma
   \geq 0,
\end{equation}
for any fixed $K \geq 1$, $\delta > 0$ desired. This is equivalent to requiring that
\begin{equation}
\label{eq:NEC}
\kappa(s) \leq  \beta \memK(s), \quad \textrm{ a.e. } s \geq 0,
\end{equation}
where we may take $\beta = \frac{K}{\delta}$.\indeed{
Indeed suppose that \eqref{eq:memkerneldecayHard} holds for some $K \geq 1$,
$\delta > 0$.  Then for any $s \geq 0$
$$
   \kappa(s) = \int_{s}^{\infty} \memK(\sigma) d \sigma
                = \int_{0}^{\infty} \memK(\sigma + s) d \sigma
                \leq K \memK(s) \int_0^\infty e^{-\delta \sigma} d\sigma
                = \frac{K}{\delta} \memK(s).
$$
Now assume \eqref{eq:NEC}.  Given $\tau > 0$ we find, for almost every $s \geq 0$,
$$
  \beta \memK(s)
    \geq \int_{s}^{\infty}  \memK(\sigma) d \sigma
    \geq \int_{s}^{s + \tau}  \memK(\sigma) d \sigma
    \geq \tau \memK(s + \tau).
$$
Choose $\tau$ so that $\rho := \beta / \tau < 1$.   With this choice, we have, for all integers
$n$,
$\memK(s + n \tau) \leq \rho^n \memK(s)$.
Given any $\sigma > 0$ pick the positive integer
$n = n(\sigma)$ such that $\sigma = n \tau +r$ where $r \in [0,\tau)$.  With this choice
$$
    \memK(s + \sigma)
        \leq \rho^n \memK(s)
        = e^{n \log \rho} \memK(s)
        = e^{r \delta} e^{- \sigma \delta} \memK(s)
        \leq K e^{-\sigma \delta} \memK(s).
$$
where of course we took $\delta = -\frac{1}{\tau}\log \rho = - \frac{1}{\tau} \log (\beta/\tau)$ and $K = e^{\tau \delta}.$
}
As a byproduct of \eqref{eq:NEC},
$\kappa \in L^{1}(\R^+)$.  Using additionally that $-\kappa' = \memK$ we have
$$
\kappa(s) \leq \kappa_{0} e^{-s/\beta} = \kappa_0 e^{-(s \delta) /K} \textrm{ for every } s>0.
$$

\begin{Rmk}\label{rmk:algDecayForMu}
The condition \eqref{eq:memkerneldecayHard} is more general in comparison to many previous
works including \cite{GattiGiorgiPata}
where it was required that $K = 1$.  The `physical' significance of \eqref{eq:memkerneldecayHard}
is that it allows the treatment of a memory kernel $\memK$ that has `flat zones', i.e.  \eqref{eq:memkerneldecayHard}
allows for the consideration of a $\kappa$ with an {\it linear decay} profile.  Indeed, take
\begin{displaymath}
\memK(s) =
\begin{cases}
 \memK_{0} &\textrm{ for } s \leq t^{*},\\
  0 &\textrm{ for } s > t^{*},
\end{cases}
\end{displaymath}
where $t^{*}, \memK_{0} > 0$ are fixed constants.  Then, according
to \eqref{eq:KappaForm},
$\kappa_{0} = \memK_{0} t^{*}$ and
\begin{displaymath}
\kappa(s) =
\begin{cases}
  (1- s/t^{*}) \kappa_{0} &\textrm{ for } s \leq t^{*},\\
  0 &\textrm{ for } s > t^{*}.\\
\end{cases}
\end{displaymath}
Observe that \eqref{eq:memkerneldecayHard} is
satisfied by any $\delta >0$ by taking $K =  e^{\delta t^{*}} > 1$.  On the other hand \eqref{eq:memkerneldecayHard} is clearly
violated for $K =1$, regardless of the choice of $\delta > 0$.
\end{Rmk}

\subsubsection{The extended phase space and its functional setting}

In order to treat \eqref{eq:JeffMoment} - \eqref{eq:Jeff1InitData} as an autonomous dynamical system
we follow \cite{Dafermos1}
(and see above for extensive further references) and introduce an additional `memory' variable $\eta$,
which is defined according to
\begin{equation}\label{eq:pastHistory}
  \eta^t(s) := \int_0^s u(t -\sigma) \mbox{d}\sigma = \int_{t-s}^t u(\sigma) \mbox{d}\sigma, \;\mbox{for } t\geq
  0, \; s\geq 0.
\end{equation}
Arguing formally, we obtain from  \eqref{eq:JeffMoment} - \eqref{eq:Jeff1InitData}
the following coupled system of equations
\begin{subequations}\label{eq:Jeffext}
  \begin{align}
   \pd{t} u + u \cdot \nabla u &= \nu \Delta u
                 - \int_0^\infty \memK(s) \Delta\eta^t(s)ds + \nabla p
                 +  f
                     \label{eq:ExtMomentum}\\
                \pd{t}\eta^t(s) &= -\pd{s}\eta^t(s) + u(t),
                 \label{eq:HistEqn}\\
                                      \nabla \cdot u &= \nabla \cdot \eta = 0,
            \label{eq:ExtDivFreeCond}\\
                u(0) = u_{0}(0), \quad& \eta^{0}(s) = \int_{-s}^{0}u_{0}(\sigma)\mbox{d}\sigma,
                     \label{eq:JeffExtInitData}\\
                u_{|\partial \DO} = 0, \quad& \eta_{|\partial \DO}=0
                     \label{eq:DirechelBoundaryCond}
  \end{align}
\end{subequations}
We recall this derivation in detail
\cite{ChekrounGlattHoltz2011a}.

In order to place \eqref{eq:ExtMomentum}--\eqref{eq:DirechelBoundaryCond} in a rigorous functional
framework we next recall some classical spaces in the mathematical theory of the Navier-Stokes equations.
See e.g. \cite{Temam1} for further background.   Further spaces needed for the memory variable $\eta$
will be recalled further on below.

Take $\mathcal{U} := \{ \phi \in (C^\infty_0(\DO))^2: \nabla \cdot \phi = 0 \}$
and define
$H := cl_{L^2(\DO)} \mathcal{U}
    = \{ u \in L^2(\DO)^2: \nabla \cdot u = 0, u \cdot n = 0 \}$.
Here $n$ is the outer pointing normal to $\partial \DO$.  On
$H$ we take the $L^2$ inner product $(u,v) := \int_\DO u \cdot v d\DO$
 and associated norm $|u| := \sqrt{(u,u)}$.
The Leray-Hopf projector, $P_H$, is defined as the orthogonal
projection of $L^2(\DO)^2$ onto $H$.  At the next order let
$V := cl_{H^1(\DO)} \mathcal{U}
    = \{ u \in H^1_0(\DO)^2 : \nabla \cdot u = 0 \}$.
On this set we use the $H^1$ inner product $((u,v)) := \int_\DO \nabla u
\cdot \nabla v d\DO$ and norm $\|u\| := \sqrt{((u,u))}$.
Note that due to the Dirichlet boundary condition the Poincar\'{e}
inequality $|u| \leq \lambda_1^{-1} \|u \|$ holds for all $u \in V$.
Here, the constant $\lambda_1$ is the first eigenvalue of the stokes
operator $A$ defined in the next paragraph.
This justifies taking $\| \cdot \|$ as a norm for $V$.  We take $V'$ to
be the dual of $V$, relative to $H$ with the pairing notated by
$\langle \cdot, \cdot\rangle$.

We next define the Stokes operator $A$ which is understood as a
bounded linear map from $V$ to $V'$ via:
\begin{equation}\label{eq:defAVVprime}
  \langle Au, v\rangle = ((u,v)) \quad u,v \in V.
\end{equation}
$A$ can be extended to an unbounded operator from $H$ to $H$
according to $Au = - P_H \Delta u$ with the domain $D(A) =
H^2(\DO)^{2} \cap V$.  By applying the theory of symmetric,
compact, operators for $A^{-1}$, one can establish the existence of an
orthonormal basis $\{e_k\}_{k \geq 1}$ for $H$ of eigenfunctions of $A$. Here
the associated eigenvalues $\{\lambda_k\}_{k \geq 1}$ form an unbounded,
increasing sequence viz.
$0 < \lambda_1 < \lambda_2 \leq \ldots \leq \lambda_n \leq \lambda_{n+1} \leq \ldots$.
We shall also make use of the fractional powers of $A$.  For $u \in
H,$ we denote $u_k = (u,e_k)$.  Given $m \geq 0$ take
$D(A^{m/2}) = \left\{
    u \in H:  \sum_k \lambda_k^{m} |u_k|^2 < \infty \right\}$
and define
$A^{m/2} u  = \sum_k \lambda_k^{m/2} u_k e_k, \quad u \in D(A^m)$.
We equip $D(A^{m/2})$ with the norm
$|u|_m := |A^{m/2} u| = \left( \sum_k \lambda_k^{m} |u_k|^2 \right)^{1/2}$.
Note that $V = D(A^{1/2})$ and that for $u \in D(A^{1/2})$,  $\| u \| = | u|_{1}$.

The nonlinear portion and part of the pressure gradient in \eqref{eq:Jeffext} is captured
in the bilinear form:
\begin{equation}\label{eq:nonlinFnDef}
  B(u,v) := P_H (u \cdot \nabla) v  = \sum_{j = 1, 2} P_H (u_j \pd{j} v)
  \quad  u \in V,  \; v \in D(A).
\end{equation}
For notational convenience we will sometimes write $B(u) := B(u,u)$.
Note that $B$ is also well defined as a continuous map from $V \times V$ to
$V'$ according to $\langle B(u,v),w \rangle  := \int_{\DO} (u \cdot \nabla) v \cdot w
d\DO = \sum_{j, k= 1}^2 \int_{\DO} u_j \pd{j} v_k w_k d\DO$.
It is easy to show that
\begin{displaymath}
      \langle B(u, v), v \rangle = 0, \quad \textrm{ for all } u,v \in V
\end{displaymath}
Classically, with H\"older's inequality and Sobolev embeddings we have the estimate
    \begin{equation}\label{eq:BestWeak}
      \left| \langle B(u,v),w \rangle \right|
      \leq C |u|^{1/2}\|u\|^{1/2}\|v\| |w|^{1/2}\|w\|^{1/2}, \quad
      \textrm{ for all }  u,v,w \in V.
    \end{equation}
On the other hand, slightly different estimates along the same lines lead to,
\begin{equation}\label{eq:BestStrg}
      \left| (B(u,v),w) \right| \leq C
          |u|^{1/2}\|u\|^{1/2} \|v\|^{1/2}|Av|^{1/2}|w|,
          \quad
          \textrm{ for every } u \in V, v \in D(A), \textrm{ and } w \in H.
\end{equation}
See e.g. \cite{Temam1} for further details.

We next introduce the main
functional spaces and operators associated with the memory
terms in \eqref{eq:ExtMomentum}, \eqref{eq:HistEqn}.  For any $m \geq 0$
we define the Hilbert spaces
\begin{equation}\label{eq:memorySpace}
  \begin{split}
  \mathcal{M}_m := L^2(\mathbb{R}^+, \memK; D(A^{m/2}))
                 = \left\{\eta : \eta \textrm{ is $(\mathbb{R}^+, D(A^{m/2}))$ measurable},
                  \int_0^\infty \memK(s) |\eta(s)|^2_{m} ds < \infty \right\},  \\
  \end{split}
\end{equation}
which are equipped with the inner products,
\begin{equation}\label{eq:L2Memory}
  [\eta,\rho]_m := \int_0^\infty \memK(s) (\eta,\rho)_m \; \mbox{d} s =
  \int_0^\infty \memK(s)\int_{\mathcal{D}}A^{m/2}\eta \cdot
A^{m/2} \rho \; \mbox{d} x \; \mbox{d} s
\end{equation}
where $A$ denotes the Stokes
operator introduced above according to \eqref{eq:defAVVprime} along with its
corresponding fractional powers recalled there.
For notational convenience and in accordance with the classical
notations for the Navier-Stokes equations we shall frequently denote $[\cdot, \cdot] =
[\cdot,\cdot]_1$ and $[[\cdot, \cdot]] = [\cdot, \cdot]_2$, which are the appropriate
norms for `weak' and `strong' estimates respectively.
See Remark~\ref{rmk:DiffWithCompEmbedding}
below.

In view of term involving $-\pd{s}$ appearing in \eqref{eq:HistEqn}
we next take, again for any $m \geq 0$,
\begin{equation}\label{eq:TDspaces}
  \mathcal{N}_m := H^1_{0}(\mathbb{R}^+, \memK; D(A^{m/2})) =
   \left\{\eta \in \mathcal{M}_m: \pd{s}\eta \in \mathcal{M}_m, \eta(0) = 0  \right\}.
\end{equation}
On $\mathcal{M}_m$ we define an (unbounded) operator $T$ according to
\begin{equation}\label{eq:Toperator}
  T\eta = -\pd{s} \eta, \quad \eta \in \mathcal{N}_m (= \mbox{Dom}(T)).
\end{equation}
We now summarize some basic properties of $T$ used in the sequel.  See e.g. \cite{GrasselliPata2002,
ContiPataSquassina, ChekrounGlattHoltz2011a}
for further details.
\begin{Lem}\label{thm:memoryBasics}
Assume that the memory kernel $\memK \in L^{1}(\R^{+})$ is nonnegative, non-increasing and
satisfies \eqref{eq:memkerneldecayHard}. Fix $m \geq 0$ and define $T$ according to \eqref{eq:Toperator}.
Then:
\begin{itemize}
\item[(i)] For any for $\eta \in \mathcal{N}_m$.
    \begin{equation}\label{eq:memEstHardKer}
      [T \eta, \eta]_m \leq 0.
    \end{equation}
      \item[(ii)] Suppose that $\eta_0 \in \mathcal{M}_m$ and that $\xi \in L^1_{loc}([0,\infty); D(A^{m/2}))$.
  Then there exists a unique mild solution $\eta \in C([0,\infty); \mathcal{M}_m)$
  of the system:
  \begin{equation}\label{eq:LonelyMemory}
      \pd{t} \eta^t = T \eta^t + \xi(t), \quad \eta^{0} = \eta_{0}.
  \end{equation}
  Moreover $\eta$ has explicit representation (cf. \eqref{eq:pastHistory}):
  \begin{equation}\label{eq:repres_formula}
\eta^t(s)=
\begin{cases}
\int_{0}^s \xi (t-\sigma)d\sigma, & \textrm{ if } 0<s\leq t,\\
\eta_{0}(s-t) +\int_{0}^t \xi (\sigma) d\sigma & \textrm{ if } s>t ,
\end{cases}
\end{equation}
which is valid for any $t>0$.
  \end{itemize}
\end{Lem}
\begin{Rmk}\label{rmk:EasyLifeWithEasyKer}
In the case where $K =1$ in \eqref{eq:memkerneldecayHard} it is not hard to show that
$[T\eta, \eta] \leq - \delta/2$.  This significantly simplifies the estimates below in e.g. \eqref{eq:eeTot1},
\eqref{eq:AsymDecayS1}.
Indeed, it is only with the functional described
in Lemma~\ref{thm:DispFunctionalProperties} (see \cite{GattiMiranvillePataZelik2008})
that the needed dissipativity properties for the
memory variable $\eta$ can be achieved.
\end{Rmk}
We next define the product spaces which will serve as the phase
spaces associated with the system
\eqref{eq:Jeffext}.
For $m \geq 0$ let $\mathcal{H}_m := D(A^{m/2}) \times
\mathcal{M}_{m+1}$ which we endow with the product norm:
$\| (u,\eta) \|_m^2 := |u|_m^2 + [\eta]_{m+1}^2$.
We will sometimes write $x =
(u,\eta) \in \mathcal{H}_m$ and denote the underlying projection
operators by $P x = u$ and $Q x = \eta$.
In view of the above conventions, to simplify notations we will
we will use often $\mathcal{H} :=
\mathcal{H}_0$ and $\mathcal{V} := \mathcal{H}_1$ and let
\begin{equation}\label{eq:typicalTotalNorm}
  \|(u,\eta)\|_0^2 = |u|^2 + [\eta]^2, \; \mbox{and}\;   \|(u,\eta)\|_1^2 = \|u\|^2 +
  [[\eta]]^2.
\end{equation}
\begin{Rmk}\label{rmk:DiffWithCompEmbedding}
The additional degree of `regularity' in the memory
space is dictated by basic the structure of system \eqref{eq:Jeffext}.  See, for example, the
estimates \eqref{eq:me1}, \eqref{eq:eeS21} in the proof of Proposition~\ref{thm:JeffGlobalAttractor}
below.

For the each of the spaces $\mathcal{M}_m$ we do not have the
compact embedding of $\mathcal{M}_{m+1}$ into $\mathcal{M}_{m}$.
Thus, in contrast to $H$ and $V$,  $\mathcal{V} \not \Subset \mathcal{H}$.   Of course this
introduces additional complications for the proof of the existence
of a global attractor in the extended phase for \eqref{eq:Jeffext} and leads to the
introduction of still further spaces.  See Lemma~\ref{thm:cmptnessLem} below.
\end{Rmk}

With these definitions in hand we may now capture the memory term in
\eqref{eq:Jeffext} a bounded operator $M: \mathcal{M}_m \mapsto D(A^{m/2})$
defined, for $m \geq 0$ via,
\begin{equation}\label{eq:memKerTerm}
  M(\eta) := \int_0^\infty  \memK(s) \eta(s) ds, \quad \eta \in \mathcal{M}_m.
\end{equation}
This definition is justified since, cf. \eqref{eq:memkerneldecayHard},
$\left|\int_0^\infty \memK(s) \eta(s) \mbox{d}s  \right|_k
     \leq \int_0^\infty \memK(s)^{1/2} \memK(s)^{1/2} \left|\eta(s)   \right|_k \mbox{d}s
     \leq C_\memK [\eta]_k$.

With all of these preliminaries in hand, using \eqref{eq:defAVVprime},
\eqref{eq:nonlinFnDef}, \eqref{eq:Toperator}, \eqref{eq:LonelyMemory}
we now rewrite the
system \eqref{eq:Jeffext} in the abstract form
\begin{equation}\label{eq:JeffAbsForm}
  \begin{split}
    \pd{t} u + \nu A u + M( A \eta) + B(u)  &=  F\\
    \pd{t} \eta &= T \eta + u\\
           (u(0),\eta^0) &= (u_0, \eta_0)\\
  \end{split}
\end{equation}
where $F = P_{H} f$.
As in \cite{GattiGiorgiPata} we may now associate a
dynamical system $\{S(t)\}_{t \geq 0}$ with
\eqref{eq:JeffAbsForm} as follows:
\begin{Prop}\label{thm:DynSystemJeff}
  Suppose that $F \in V'$ and $x_{0} = (u_{0}, \eta_{0}) \in \mathcal{H}$.  Then there exists a
  unique $x = (u,\eta) \in C([0, \infty), \mathcal{H})$ with $x(0) = x_{0}$, such
  that $Px = u \in L^{2}_{loc}([0,\infty),V)$ and
  satisfying \eqref{eq:JeffAbsForm}
  for every $t > 0$ (in the appropriate weak sense).
By taking, for any $t \geq 0$, $S(t) x_{0} = x(t) = (u(t), \eta^t)$ we may thus define a continuous semigroup
$\{S(t)\}_{t \geq 0}$ on the phase space $X = \mathcal{H}$ (endowed with its usual topology).
\end{Prop}
\noindent This Proposition may be established along classical lines with a Galerkin approximation.
See e.g. \cite{Temam1} and additionally \cite{GiorgiGrasselliPata99, GattiGrasselliPata2004} for further details on Galerkin
schemes in the history space context.

\subsubsection{The global attractor, invariant measures and related estimates}
\label{sec:GlobalAttractorJeff}
With this basic framework now in place we proceed to prove
that the dynamical system $\{S(t)\}_{t \geq 0}$ defined by
\eqref{eq:JeffAbsForm}, in the sense made precise by Proposition~\ref{thm:DynSystemJeff},
possesses a global attractor so that we may with no further efforts apply
the results in Section~\ref{sec:MainResults} to \eqref{eq:JeffAbsForm}.  More precisely
we prove:
\begin{Thm}\label{thm:JeffGlobalAttractor}
Suppose in addition to the conditions imposed in Proposition~\ref{thm:DynSystemJeff} that
$F \in H$.  Then the dynamical system $\{S(t)\}_{t \geq 0}$ defined on $\mathcal{H}$ by \eqref{eq:JeffAbsForm}
according to Proposition~\ref{thm:DynSystemJeff} has a connected global attractor $\mathcal{A}$
and hence the results in Theorems \ref{thm:Prop1}, \ref{thm:Prop2}
hold for $\{S(t)\}_{t \geq 0}$ with $(X,d) = (\mathcal{H}, \| \cdot \|)$.
\end{Thm}
The rest of this section is devoted to the proof of Theorem~\ref{thm:JeffGlobalAttractor}.
\noindent The proof consists in verifying each of the two conditions imposed
by Proposition~\ref{thm:AttractorCriteria}.  With the sufficient condition supplied by this
Proposition in mind, we proceed
in two step.  At the first step we use energy type estimates to establish the existence of an
absorbing set $\mathfrak{B}$.  Note that, in view of \eqref{eq:memEstHardKer}, the needed
dissipation in variable $\eta$ does not follow directly from these standard estimates.
As such we employ an additional functional as dictated by Lemma~\ref{thm:DispFunctionalProperties}
introduced below.
For the second step we achieve the splitting suggested by Proposition~\ref{thm:AttractorCriteria}
by considering solution operators $S_{1}$ associated with no external forcing $F$ and $S_{2}$ with
zero initial data.  See \eqref{eq:split1}, \eqref{eq:split2} respectively.
Since, as mentioned in Remark~\ref{rmk:DiffWithCompEmbedding}, $\mathcal{V}$ is not compactly embedded
in $\mathcal{H}$, we are required to introduce still further spaces at achieve the compactness
desired for $S_{2}$.   This is the significance of the space $\mathcal{E}$ in \eqref{eq:compembSpaceNorm}
and Lemma~\ref{thm:cmptnessLem} immediately following.  Note that the formal estimates that follow
may be rigorously justified in the context of a suitable Galerkin approximation scheme for \eqref{eq:JeffAbsForm}.

{\bf Step 1: Energy estimates and dissipativity.}
By multiplying the first equation in \eqref{eq:JeffAbsForm}
with $u$ and making use use of the cancellation properties of $B$ we infer
\begin{equation}\label{eq:ee1}
    \frac{1}{2} \frac{d}{dt} |u|^2 + \nu \|u\|^2 + [ \eta, u ]
           = \langle F, u \rangle.
\end{equation}
We next multiply the second equation  in \eqref{eq:JeffAbsForm} by $A\eta$ and deduce:
\begin{equation}\label{eq:me1}
  \frac{1}{2} \frac{d}{dt} [\eta]^2 = [T\eta, \eta] + [\eta, u].
\end{equation}
Adding \eqref{eq:ee1} and \eqref{eq:me1}, using \eqref{eq:memEstHardKer} and
making an obvious application of Young's inequality we have
\begin{equation}\label{eq:eeTot1}
  \frac{d}{dt} (|u|^2 + [\eta]^2)
     + \nu \|u\|^2
           \leq C_\nu |F|_{V'}^2.
\end{equation}
This inequality does not evidence any dissipation
in the `memory variable' $\eta$.  We are therefore unable to directly
invoke the Poincar\'{e} inequality and
the Gronwall lemma in order to conclude the existence of a bounded
absorbing set as we would for e.g. the classical 2-D Navier-Stokes equations.
See however Remark~\ref{rmk:EasyLifeWithEasyKer} above.

To overcome this difficulty we recall the following Lemma (see
\cite{GattiMiranvillePataZelik2008} and also
\cite{ChekrounDiPlinioGlattHoltzPata2010, ChekrounGlattHoltz2011a})
which we will use at this step and later on in Step 2 when we
introduce  a splitting of $S(t)$ (cf. Proposition
\ref{thm:AttractorCriteria}, above).
\begin{Lem}\label{thm:DispFunctionalProperties}
Suppose that $\memK$ and $\kappa$ satisfy the conditions imposed in \eqref{eq:KappaForm}, \eqref{eq:memkerneldecayHard}.
Fix any $m \geq 0$ and assume that $\eta_0 \in \mathcal{M}_m$, $\xi \in L^1_{loc}([0,\infty); D(A^{m/2}))$
and that $\eta$ is the unique solution of \eqref{eq:LonelyMemory}
corresponding to $\xi$ and $\eta_{0}$ as guaranteed by
Lemma~\ref{thm:memoryBasics}, (ii).  Define, the
functional
    \begin{equation}\label{eq:fountainReachAround}
      \Gamma_m(t)= \Gamma_m(\eta^t)
      := \int_0^\infty \kappa(s) |
      \eta^{t}(s) |^2_m \mbox{d}s.
    \end{equation}
    Then $\Gamma_m(t)$ satisfies the following differential
    inequality,
    \begin{equation}\label{eq:fountainInequality}
      \frac{d\Gamma_m (t)}{dt}
        + \frac{1}{4\beta} (\Gamma_m (t) + \beta [\eta^t]^2_m) \leq
        2\beta^2 \|\memK\|_{L^1(\mathbb{R}^+)} | \xi(t) |^2_m,
    \end{equation}
    where $\beta$ is the constant arising in
    \eqref{eq:NEC}.
\end{Lem}
We now introduce, for each $\Gadol > 0$, the following functional,
\begin{equation}\label{eq:compinstatorFn}
  \Phi^{0}_{\Gadol}(t)
    := \Gamma_{1}(\eta^t) + \Gadol( |u(t)|^2 + [\eta^t]^2).
\end{equation}
Applying Lemma~\ref{thm:memoryBasics}
with $k = 1$ to the second equation in \eqref{eq:JeffAbsForm}
we find that
$$
 \frac{d}{dt} \Gamma_{1} + \frac{1}{4 \beta} (\Gamma_{1} + \beta [\eta]^{2} )
 \leq C_{\memK}\|u\|^{2},
 $$
where $C_{\memK} :=2\beta^2 \|\memK \|_{L^1(\mathbb{R}^+)}$, as arising
in \eqref{eq:fountainInequality}.  Combining this observation with \eqref{eq:eeTot1},
we estimate
\begin{equation}\label{eq:PhiGimelIneqGen}
 \frac{d}{dt}  \Phi^{0}_{\Gadol} + \frac{1}{4 \beta} (\Gamma_{1} + \beta [\eta]^{2} )
    \leq (C_{\memK} - \Gadol \nu)  \|u\|^2
           + \Gadol C_\nu |F|_{V'}^2.
\end{equation}
Take
\begin{equation}\label{eq:gimelandGammaCons}
\Gadol := \frac{1}{4 \lambda_1 \nu} + \frac{C_\memK}{\nu} =
\frac{1}{4 \lambda_1 \nu} + \frac{2\beta^2 \|\memK \|_{L^1(\mathbb{R}^+)}}{\nu},
\quad \gamma
:= \frac{1}{4 \max \left\{ \beta, \Gadol \right\}}
\end{equation}
where we recall that $\lambda_1$  the constant in the
Poincar\'e inequality and $\beta$ comes from \eqref{eq:NEC}.
With this choice of $\Gadol$ and associated $\gamma$, we observe that
\indeed{
    Indeed, we must consider two cases.  If $\beta \leq \Gadol$ then
    $$
    \frac{1}{4\beta} \Phi^{0}_{\beta} =
    \frac{1}{4\beta} \Gamma_1 + \frac{1}{4}( |u(t)|^2 + [\eta^t]^2)
    \geq \frac{1}{4 \Gadol} \Gamma_1 +\frac{1}{4} ( |u(t)|^2 + [\eta^t]^2)
    =   \frac{1}{4 \Gadol} \Phi^{0}_{\Gadol}.
    $$
    On the other hand if $\beta > \Gadol$
    $$
    \frac{1}{4\beta} \Phi^{0}_{\beta} \geq
    \frac{1}{4\beta} (\Gamma_1 + \Gadol ( |u(t)|^2 + [\eta^t]^2)
    =   \frac{1}{4 \beta} \Phi^{0}_{\Gadol}.
    $$
}
$$
    \frac{1}{4\beta} \Phi^{0}_{\beta} \geq \gamma \Phi^{0}_{\Gadol}.
$$
and so with \eqref{eq:PhiGimelIneqGen} and \eqref{eq:gimelandGammaCons}
\begin{displaymath}
   \frac{d}{dt} \Phi^{0}_{\Gadol} + \gamma \Phi^{0}_{\Gadol}
   \leq C_{\nu, \lambda_{1}, \memK} |F|^{2}_{V'},
\end{displaymath}
With the semigroup notation $x(t) = S(t) x_{0}$, $x_{0} = (u_{0}, \eta_{0})$
we therefore estimate using the Gr\"onwall  inequality that
\begin{equation}\label{eq:bndBallL2Calc}
   \|x(t)\|^{2}_{0} 
        \leq C_{\nu, \lambda_{1},\memK} \left(e^{-\gamma t} \Phi^{0}_\Gadol (0) + \int_{0}^{t} e^{-\gamma (t-s)} |F|_{V'}^{2}\right)
    \leq C_{\nu, \lambda_{1},\memK} (e^{-\gamma t}  \|x_{0}\|^{2}_{0} +   |F|_{V'}^{2}).
\end{equation}
Note that we have used \eqref{eq:NEC} with \eqref{eq:compinstatorFn} to infer the last inequality.
We may hence infer the existence of an absorbing ball $\mathfrak{B}_{0}$ in $\mathcal{H}$ whose
radius depend only on $|F|_{V'}$ and the universal $C =C_{\nu, \lambda_{1},\memK}$
in the final inequality above. More explicitly, \eqref{eq:bndBallL2Calc} shows that,
 given any $K > 0$,  there exists a time $t^*_1= t^*_1(K)$ such that
for $t > t^*_1$:
\begin{equation}\label{eq:absorbDef1}
  S(t) \{x_{0} \in \mathcal{H} : \|x_{0}\| \leq K \} \subset \mathfrak{B}_{0}.
\end{equation}
This is the first item required for Proposition~\ref{thm:AttractorCriteria}.  We turn to the second
step.

{\bf Step 2: The splitting property.}
We next exhibit a splitting of $S(t) = S_1(t) + S_2(t)$ taking the form
required by Proposition~\ref{thm:AttractorCriteria}.
Given arbitrary initial data $x_0 = (u_0,\eta_0) \in \mathcal{H}$ we
define $(u_1(t), \eta_1^t) := S_1(t)x_0$ to be the solution at time $t$ of:
\begin{equation}\label{eq:split1}
  \begin{split}
    \pd{t} u_1+ \nu A u_1 + M( A \eta_1)  &= - B(u, u_1), \\
    \pd{t} \eta_1 &= T \eta_1 + u_1,\\
           (u_1(0),\eta_1^0) &=  (u_0, \eta_0).\\
  \end{split}
\end{equation}
We take $(u_2(t), \eta_1^t) = S_2(t) x_0$ to be the solution of:
\begin{equation}\label{eq:split2}
  \begin{split}
    \pd{t} u_2 +\nu A u_2 + M( A \eta_2) &= F - B(u,u_2),\\
    \pd{t} \eta_2 &= T \eta_2 + u_2,\\
           (u_2(0),\eta_2^0) &= (0,0).\\
  \end{split}
\end{equation}
Note that the $u$ appearing in the nonlinear terms of both \eqref{eq:split1}, \eqref{eq:split2}
is the first component of the solution of
\eqref{eq:JeffAbsForm}; in other words $u= P S(t) x_0$. Clearly, given any $x_0 \in \mathcal{H}$, we have $S(t)x_0 = S_1(t)x_0 + S_2(t) x_0$.

The estimates for $S_{1}$, in view of Proposition~\ref{thm:AttractorCriteria}, (ii)
are carried in $\mathcal{H}$ and are very similar to those in Step 1 for $S$.    Indeed following
the arguments leading up to \eqref{eq:eeTot1} we obtain
\begin{equation}\label{eq:AsymDecayS1}
   \frac{d}{dt} (|u_{1}|^2 + [\eta_{1}]^2)
     + \nu \|u_{1}\|^2
                \leq 0.
\end{equation}
We now make a second application of Lemma~\ref{thm:DispFunctionalProperties}.
Defining, for $\Gadol > 0$, $\Phi^{1}_{\Gadol}(t) := \Gamma_{1}(\eta^t_1) + \Gadol( |u_1(t)|^2 + [\eta^t_1]^2)$,
arguing as above in \eqref{eq:PhiGimelIneqGen} and tuning $\Gadol$ and $\gamma$ as in \eqref{eq:gimelandGammaCons} we infer
\begin{equation}\label{eq:bndBallL2CalcS1}
   \|x_{1}(t)\|^{2}_{0} 
    \leq C_{\nu, \lambda_{1},\memK} (e^{-\gamma t}  \|x_{0}\|^{2}_{0}),
\end{equation}
where $x_{1}(t) = S_{1}(t)x_{0}$.
We infer that for every $K > 0$:
\begin{equation}\label{eq:decayDetCase}
  \sup_{\|x_0\|_0 \leq K}\|S_1(t)x_0\|_0 \xrightarrow{t \rightarrow \infty}  0.
\end{equation}
so that indeed $S_{1}$ plays the desired role required for Proposition~\ref{thm:AttractorCriteria}, (ii).

In remains to establish the requirement on $S_2$ given in Proposition~\ref{thm:AttractorCriteria}, (ii);
we must show that, for each $t > 0$, and each bound subset $B$ of $\mathcal{H}$ that $S_{2}(t)B$
is a precompact subset of  $\mathcal{H}$.
To this end we first show that $S(t)B$ is a bounded subset of $\mathcal{V}$
and then, owing to the complication that $\mathcal{V}$ is not compactly embedded in $\mathcal{H}$,
we must take further steps.  See Remark~\ref{rmk:DiffWithCompEmbedding} and
Lemma~\ref{thm:cmptnessLem} below.

In order to have an equation for $t \mapsto \|S_{2}(t)x_{0}\|_{1}$ we multiply the first equation of
\eqref{eq:split2} by $Au_2$ and the second equation by $A^2 \eta_2$.  With \eqref{eq:memEstHardKer}
one finds that
\begin{displaymath}
  \frac{d}{dt} (\|u_2\|^2 + [[\eta_2]]^2) + \nu|Au_2|^2 \leq 2 |(B(u,u_2),Au_2)| + C|F|^2.
\end{displaymath}
By applying standard estimates on the nonlinear terms, \eqref{eq:BestStrg}, we find:
\begin{displaymath}
  \begin{split}
  2|(B(u,u_2), Au_2)| \leq C |u|^{1/2}\|u\|^{1/2}\|u_2\|^{1/2} |Au_2|^{3/2}
                 \leq C_\nu |u|^2\|u\|^2\|u_2\|^2 + \nu|Au_2|^{2}.\\
  \end{split}
\end{displaymath}
Hence:
\begin{equation}\label{eq:eeS21}
  \begin{split}
    \frac{d}{dt} (\|u_2\|^2 + [[\eta_2]]^2)
   \leq  C_{\nu} (|u|^2\|u\|^2\|u_2\|^2 + |F|^2)
   \leq  C_{\nu} (|u|^2\|u\|^2(\|u_2\|^2 + [[\eta_2]]^2) + |F|^2).
  \end{split}
\end{equation}
Let $\Upsilon(t) = C_{\nu} \int_0^t |u|^2\|u\|^2$ where $C_{\nu}$ is precisely the constant
appearing in \eqref{eq:eeS21}.
Returning to \eqref{eq:eeTot1} we may observe that for any $x_{0} = (u_{0}, \eta_{0})$, we have
$\sup_{s \leq t} |u|^{2} + \nu \int_{0}^{t} \| u\|ds \leq  |u_{0}|^{2} +  t C_{\nu} |F|_{V'}$,
so that clearly $e^{\Upsilon(t)} < \infty$ for each $t > 0$.  Applying this functional to \eqref{eq:eeS21},
integrating in time appropriately, and noting that $S_{2}(0)x_{0} \equiv 0$ we find that
\begin{equation}\label{eq:mathcalVBndS2}
  \begin{split}
   \sup_{0 \leq r \leq t} \|S_{2}(t)x_{0} \|_{1} = \sup_{0 \leq r \leq t} (\|u_2(r)\|^2 + [[\eta_2^r]]^2)
      \leq C_{\nu} \int_0^t e^{\Upsilon(t) -\Upsilon(s)}|F|^2ds
      \leq C_{\nu, t} |F|^2.
  \end{split}
\end{equation}
This shows that, for every bounded set $B \subset \mathcal{H}$ and for every $t \geq 0$
that $S(t)B$ is a bounded subset of $\mathcal{V}$.    However, as described above
in Remark~\ref{rmk:DiffWithCompEmbedding}, $\mathcal{V}$ is not compactly embedded
in $\mathcal{H}$ and further steps are therefore required to established the desired compactness
of $S_{2}$.

To compensate for this difficulty we follow previous works (see e.g.
\cite{GattiGiorgiPata, ContiPataSquassina}) and
 introduce some additional spaces.
On $\mathcal{M}_1$ we define `tail functional'
\begin{equation}\label{eq:tailNormMemSpace}
\mathfrak{T}_{\memK}(\eta) = \sup_{\sigma \geq 1} \; \sigma \cdot \int_{(0,1/\sigma) \cup (\sigma, \infty)}
  \memK(s)\|\eta(s)\|^2\mbox{d}s,
\end{equation}
and consider the subspace
\begin{equation}\label{eq:compembSpaceNorm}
 \mathcal{E} = \{\eta \in \mathcal{N}_2: \mathfrak{T}_{\memK}(\eta) <
  \infty\} \subset \mathcal{N}_2,
\end{equation}
which we endow with the norm,
$| \eta |_{\mathcal{E}}^2 := [[ \eta ]]^2 + [T \eta]^2 +
  \mathfrak{T}_{\memK}(\eta)$.
Under these definitions $\mathcal{E}$ may be shown to be a closed subset
of $\mathcal{M}_{2}$ (relative to $| \cdot |_{\mathcal{E}}$) and hence is a Banach space; see \cite{PataZucchi1}.
We thus define the product space $\tilde{\mathcal{V}} = V \times \mathcal{E} \subset \mathcal{V}$
and endow $\tilde{\mathcal{V}}$ with the
norm,
\begin{equation}\label{eq:tildeEnorm}
  \| (u, \eta) \|_{\tilde{\mathcal{V}}}^2 := \| u \|^2 + | \eta |_{\mathcal{E}}^2
    = \| (u, \eta) \|_{1}^{2} +  [T \eta]^2 +
    \mathfrak{T}_{\memK}(\eta).
\end{equation}
We have the following compactness results and related estimates on the evolution equation
for the history variable.  See e.g.
\cite{GattiGiorgiPata, ContiPataSquassina, ChekrounGlattHoltz2011a}.
\begin{Lem}\label{thm:cmptnessLem}
  Assume that the memory kernel $\memK \in L^{1}(\R^{+})$ is nonnegative, non-increasing and
satisfies \eqref{eq:memkerneldecayHard}.
  \begin{itemize}
  \item[(i)] $\mathcal{E}$ is compactly embedded in $\mathcal{M}_1$
    and hence $\tilde{\mathcal{V}}$ is compactly embedded in $\mathcal{H}$.
  \item[(ii)] Suppose that $\eta_{0}  = 0$\footnote{With some modifications we may assume merely that
  $\eta_0 \in \mathcal{E}$.  Since this is unneeded for our current purposes we do not state the Lemma in this greater generality here.}, and that $\xi \in L^\infty_{\mbox{loc}}((0,\infty); V)$.  Let $\eta^t$ be the unique mild
  solution in $C([0,\infty), \mathcal{M}_1)$ of,
    \begin{equation}\label{eq:memorequation}
      \pd{t} \eta^t = T \eta^t + \xi(t), \quad \eta^{0} = 0,
    \end{equation}
        guaranteed by Lemma~\ref{thm:memoryBasics}, (iii).  Then
    there exists a universal positive constant $C = C_{\memK}$
    (depending on $\memK$ but independent of $t$ and the data)
    such that
    \begin{equation}\label{eq:tailnormBndsConcl}
    [T \eta^t]^{2} + \mathfrak{T}_{\memK}(\eta^t) \leq C \sup_{s \in [0,t]} \|\xi(s)\|^{2} < \infty.
    \end{equation}
  \end{itemize}
\end{Lem}

With Lemma~\ref{thm:cmptnessLem} in hand and having already established
\eqref{eq:mathcalVBndS2} we finally proceed to show that $S_{2}$ is compact.
Fix any $x_{0} = (u_{0}, \eta_{0}) \in \mathcal{H}$.
Observe that, due to \eqref{eq:mathcalVBndS2} we may infer that
\begin{equation}\label{eq:TailNormApp1}
\sup_{  0\leq r \leq t} \|P S_2(t)x_{0}\|^{2}  \leq \sup_{  0\leq r \leq t} \| S_2(t)x_{0}\|_{1}^{2} := J(t)< \infty
\end{equation}
where we note that $J(t)$ is independent of $x_{0}$.   We now apply
Lemma~\ref{thm:cmptnessLem}, (ii) with $\xi = P S_2(t)x_{0}$ in \eqref{eq:memorequation}.
From \eqref{eq:split2}, \eqref{eq:tailnormBndsConcl} and \eqref{eq:TailNormApp1} we infer
\begin{equation}\label{eq:TailNormApp2}
[T Q S_{2}(t) x_{0}]^{2} + \mathfrak{T}_{\memK}( Q S_{2}(t) x_{0}) < C J(t)
\end{equation}
where $C$ is the constant appearing in \eqref{eq:tailnormBndsConcl} and is
independent of $x_{0}$.  Combining \eqref{eq:TailNormApp1}, \eqref{eq:TailNormApp2}
and cf. \eqref{eq:tildeEnorm}
$$
   \sup_{ \|x_0\|_0 \leq K} \|S_2(t) x_0 \|_{\tilde{\mathcal{V}}}^2 \leq C J(t) < \infty,
$$
for any $K > 0$.
In other words we infer that for every bounded set $B \subset \mathcal{H}$ and for every $t \geq 0$
that $S_{2}(t)B$ is a bounded subset of $\tilde{\mathcal{V}} \subset \mathcal{V}$ and so conclude
from Lemma~\ref{thm:cmptnessLem}, (i) that $S_{2}$ in compact.
The second requirement for Proposition~\ref{thm:AttractorCriteria}
has therefore been fulfilled and the proof of Proposition~\ref{thm:JeffGlobalAttractor} is finally complete.

\subsection{Neutral delay differential equations}
\label{sec:NDDEs}

A retarded functional differential equation (RFDE) describes a system where the rate of change
of its state is determined by the present and the past states of the system. If, additionally, the rate of change of
the state depends on the rate of change of the state in the past,
the system is called a neutral functional differential
equation (NFDE). When only discrete values of the past have influence on the present rate of
change of the system's state, the corresponding mathematical model is either a delay differential equation (DDE)
or a neutral differential equation (NDDE). The theory of RFDEs and NFDEs is both of theoretical and practical
interest, as these types of models provide a powerful framework
used in the study of many phenomena in the applied sciences;
for example in physics,
biology, economics, and control theory to name a few.

The initial development of neutral functional differential equations
has been intimately related to some particular transmission line
problems modeled by hyperbolic PDEs such as the telegrapher's
equations. In the seminal works \cite{Abolina_Myskis60},
\cite{Brayton_Miranker64}, it was observed that certain linear
hyperbolic PDEs with {\it (non)linear boundary conditions} were
equivalent to an equation of the form $
\dot{x}_t-\mathfrak{L}\dot{x_t}=f(x_t)$ where $\mathfrak{L}$ is an
operator which does not depend upon the values of $x(t)$ but only on
values in the past. In other words the evolution equation of the
system may be characterized by the simultaneous presence of delayed
and non-delayed derivatives.  We will explore this relationship in
further detail for the specific case of the telegrapher's equation,
with simplified, linear, boundary conditions \eqref{Eq_BC2}, below in
subsection~\ref{Sec_NDDE_HyperPDEs}.

The linear boundary conditions \eqref{Eq_BC2} have been chosen here
merely to give a simple illustration of how NDDEs may arise in
problems modeled by hyperbolic PDEs. More complex, physically
realistic boundary conditions, such as nonlinear or dynamic boundary
conditions, lead to many other (nonlinear) NDDEs. See e.g.
\cite{HaleLunel93, Blakely_Corron04, Barton_al07}. In
\cite{Blakely_Corron04} for example, a nonlinear NDDE model derived from the
telegrapher's equation, but with nonlinear boundary conditions, was
employed to model an actual electronic device.    Here the delay may
be seen to arise partially from the signal passage time through the
transmission line itself.  The presence of rich, chaotic
delay-induced dynamics, was  observed experimentally and shown
to be in good agreement with the numerical simulation of their
nonlinear NDDE model. More precisely  it was shown that, in the
appropriate regime, the attractor determined by this NDDE presented
a similar coarse-grained structure in comparison to the attractor
reconstructed from the empirical observation of the
electronic device they were studying. This particular NDDE has been
further analyzed in \cite{Barton_al07}. Here it was shown that
homoclinic bifurcations were at the origin of the complicated
dynamics observed in \cite{Blakely_Corron04}. 

Of course, to
determine if such complicated dynamics may be associated with the
existence of a ``complicated'' invariant probability measure is out
of reach due to the infinite dimensional  character of the phase
space $X = C([-\tau,0], \R^{n})$.\footnote{This question seems intractable 
since there is no equivalent notion of the Lebesgue measure to employ as a universal
reference measure $\mathfrak{m}_0$.  We thus have no means to
declare that a measure $\mathfrak{m}$ to be ``complicated'' in terms
of a lack of absolute continuity with respect to a reference
measure.}
Indeed, in view of the rich
chaotic dynamics observed in \cite{Blakely_Corron04, Barton_al07}, even the
question of the \emph{existence} of an invariant probability measure
for such systems seems to be nontrivial and, to the best of the
authors' knowledge, this question seems to have never been
previously addressed in the literature.

Theorem~\ref{Thm_existence_attractor} below yields a general,
abstract result concerning the existence of invariant probability
measures of NFDEs and furthermore associates these measures with
certain temporal averages. For each initial Borel probability
measure $\brlMsr_0$, we obtain an invariant measure $\mathfrak{m}$
as a generalized Banach limit of $\brlMsr_0$ through the time
average of the semigroup generated by the NDDE. Note that, by our
construction, the support of any  invariant measures obtained by
Theorem~\ref{Thm_existence_attractor} (see also Theorem
\ref{THM_IM_NDDE}) is furthermore contained in $\mathcal{A}$.  This
does not preclude the possibility that the invariant measure
$\brlMsr$ could be supported by the whole of the global attractor
with its complex geometry. Our results apply to any NDDE possessing
a global attractor $\mathcal{A}$; cf.
Theorem~\ref{thm:Prop1},~\ref{thm:Prop2} and
Definition~\ref{def:GBL} above. Note that, in this case, such a
semigroup has no smoothing effects in finite time (i.e. the
semigroup does not have a compact absorbing set) and acts on
$X=C([-\tau,0]; \mathbb{R}^{n})$ which is a non-reflexive Banach
space.  As such, the results in \cite{LukaszewiczRealRobinson2011}
do not apply for such equations.  To provide concrete
situations where NDDEs possess a global attractor,
Theorem~\ref{THM_IM_NDDE} provides a novel sufficient and
``checkable" condition which yields a broad class of interesting
examples.

The rest of this section is organized as follows.  We begin by
further detailing the relationship between NDDEs and hyperbolic PDEs
with an example involving the telegrapher's equation (with simplified
boundary conditions). In subsection~\ref{Sec_prelimNDDE} we briefly
recall some preliminaries from the general theory of NDDEs 
which allow us to establish the main abstract result,
Theorem~\ref{Thm_existence_attractor}, which establishes the existence of
invariant probability measures for a broad class of NFDEs. In
subsection~\ref{Sec_resultsNDDE} we formulate and prove a
result of a more practical nature, Theorem~\ref{THM_IM_NDDE}, which
gives concrete sufficient conditions for the 
results in Theorem~\ref{Thm_existence_attractor}. 
The final subsection applies Theorem~\ref{THM_IM_NDDE} to
some examples related to the classical NDDE model proposed in
\cite{Brayton_Miranker64}.

\subsubsection{NDDEs and hyperbolic PDEs}
\label{Sec_NDDE_HyperPDEs}
We can get some idea of how NDDEs (and also difference equations
with retarded arguments may) be derived from  certain classes of
hyperbolic PDEs by considering an example involving the
telegrapher's equation with particular linear boundary conditions in
some details. We shall see that the type of retarded system derived
depends on the boundary conditions. For this purpose, we recall the
classical problem considered in \cite{Brayton_Miranker64} of the
determination of the current $I$ and voltage $V$ in a transmission
line terminated at each end by linear resistors or nonlinear circuit
elements such as diodes. Note  that we adopt here a different
approach than the one classically presented in \cite{HaleLunel93}.
The reasons of this, which are related to stability and regularity
issues, will be clarified in the course of the forthcoming
discussions.  In particular, the exposition adopted here allows
us to illustrate an important relationship between
difference operators and NDDEs.  As we shall see,
cf. Lemma \ref{Lem_decomp_NDDE} below,
the structure of a certain difference operator associated with
the type of NDDEs (or NFDEs) we consider here 
determines the non-compact nature of the 
associated semigroup.

Let us take the $x-$axis in the direction of the line, with ends at $x=0$ and $x=1.$
It is assumed that the voltage $V$ and current $I$
at each point along the line are governed by the following telegrapher's equations
\bea \label{Eq_line_transmission}
 \partial_x V= &- L \partial_t I,\\
\partial_x I= & - C \partial_t V,
\eea
where $L$ and $C$ are respectively the inductance and the capacitance,
both of which will be assumed to be constant.
It is well-known that the solutions of \eqref{Eq_line_transmission} can be represented as a superposition
of traveling waves moving towards $x=0$ and $x=1$ with velocity $c=(LC)^{-1/2}$.
We have then for $x\in(0,1)$ and $t>0$,
\begin{equation}
\begin{split}
 \label{Eq_propagations}
V(x,t)= & \frac{1}{2}[\phi(t-c^{-1}x) +\psi(t+c^{-1}x)],\\
I(x,t)= & \frac{1}{2Lc}[\phi(t-c^{-1}x) -\psi(t+c^{-1}x)] =
\frac{\sqrt{C/L}}{2}[\phi(t-c^{-1}x) -\psi(t+c^{-1}x)],
\end{split}
\end{equation}
where $\phi$ and $\psi$  are some unknown functions to be determined from the boundary conditions.
See e.g.
\cite[Section 164]{Smirnov}.  Frequently, $\phi$ (resp. $\psi$) is called a  {\it progressive wave}
(resp. {\it reflected wave})
since its argument is a solution of the increasing (resp. decreasing)
characteristic $dt/dx=c^{-1}$ (resp. $dt/dx=-c^{-1}$);
see e.g. \cite{ZT86}.

Let us consider first the following linear boundary conditions \bea
\label{Eq_BP1} V(0,t)+R_0 I(0,t)=E,\;\; I(1,t)=0, \textrm{ for } t
\geq 0, \eea where $E$ is a voltage which is assumed to be constant
for simplicity. Physically, this corresponds to a transmission line
that is ``disconnected" from the nonlinear circuit element (diode)
at the right hand boundary, i.e.  $I(1,t)=0$, with some resistor of
resistance $R_0$ at the left boundary. Such boundary conditions are
expressed  in terms of $\phi$ and $\psi$ as \bea \label{Eq_BP-waves}
(1+R_0 \sqrt{C/L}) \phi(t) &+(1-R_0\sqrt{C/L})\psi(t)=2 E,\\
\psi(t+\sqrt{LC}) &=\phi(t-\sqrt{LC}).
\eea

 Consider the Cauchy problem associated with
\eqref{Eq_line_transmission} -- \eqref{Eq_BP1} and given initial
data $V(0, \cdot) = V_0(\cdot)$ and $I(0, \cdot) = I_0(\cdot)$. From
the theory of first order PDEs, we know  that this Cauchy problem is
well-posed provided that $(x,t)$ belongs to the triangular domain
$\mathfrak{T}$ determined by the characteristics according to,
$$
\mathfrak{T}:=\Big\{ (x,t)\in [0,1]\times \mathbb{R}^{+}:  \; -\frac{1}{c}+ \frac{x}{c}\leq t\leq \frac{x}{c}
        \mbox{ and } -\frac{x}{c}\leq t\leq \frac{1}{c}-\frac{x}{c}\Big\}.
        $$
Now, in order to determine $\phi$ and $\psi$ from the initial data and boundary conditions,
we observe that since $x\in[0,1]$, $\phi$ can be determined from $V_0$ and $I_0$ only on
$[-\sqrt{LC},0]$ according to $\phi(-\frac{x}{c})=V_0(x)+\sqrt{L/C} I_0(x)$.
Similarly we initially determine $\psi$ only on  $[0,\sqrt{LC}]$ via $\psi(\frac{x}{c})=V_0(x)-\sqrt{L/C} I_0(x)$.
It is then equation \eqref{Eq_BP-waves} which allows one to determine $\phi$ and $\psi$ for all of $\mathbb{R}$. To be more precise, let us introduce $\tilde{\psi}(t)=\psi(t+\sqrt{LC})$;
now since $\psi$ is known on  $[0,\sqrt{LC}],$ $\tilde{\psi}$ is  determined on $[-\sqrt{LC},0],$ as $\phi$ is.
We can therefore rewrite the system \eqref{Eq_BP-waves}  as the following {\it difference equation} in continuous time:
\bea \label{Eq_wavesfunction}
\begin{bmatrix} \phi(t) \\ \tilde{\psi}(t) \end{bmatrix}=\begin{bmatrix} 0 & -
\frac{1-R_0\sqrt{C/L}}{1+R_0\sqrt{C/L}}\\ 1 & 0 \end{bmatrix} \begin{bmatrix} \phi(t-\tau) \\ \tilde{\psi}(t-\tau) \end{bmatrix}+\begin{bmatrix} \frac{2 E}{1+R_0\sqrt{C/L}} \\ 0 \end{bmatrix},\\
 \eea
where $\tau=\sqrt{LC}.$ This delay $\tau$ may therefore be
interpreted as the characteristic time associated with wave
propagation; it represents the transmission time through the
transmission line itself.

The system \eqref{Eq_wavesfunction} may be written in a compact form as
\bea \label{Eq_excompact_form}
    x(t)-Bx(t-\tau)=f,
\eea where $x(t)=[\phi(t)  \; \tilde{\psi}(t)]^T$, $B$ is the matrix
arising in  \eqref{Eq_wavesfunction} and $f$ denotes the constant
vector in \eqref{Eq_wavesfunction}. By introducing  the operator
$D_0 \Phi:= \Phi(0)-B\Phi(-\tau)$  acting on functions $\Phi \in
X:=C([-\tau,0],\mathbb{R}^2)$ this equation takes the functional
representation \be\label{Eq_difference-representation} D_0 x_t=f,
\ee where  $x_t\in X$, for all $t$ and we take
$x_t(\theta)=x(t+\theta)$ for all $\theta\in[-\tau,0]$.  When
supplemented with the initial data $x_0=[\phi,\tilde{\psi}]^T,$
where $\phi(\theta)=V_0(-c\theta)+\sqrt{L/C} I_0(-c\theta)$ and
$\tilde{\psi}(\theta)=V_0(c\theta+1)-\sqrt{L/C} I_0(c\theta+1),$
with $\theta\in[-\tau,0],$ equation
\eqref{Eq_difference-representation}(along with
\eqref{Eq_propagations}) gives thus another representation of the
Cauchy problem associated with the original  boundary value problem
(BVP) \eqref{Eq_line_transmission}-\eqref{Eq_BP1}.

\begin{Rmk}\label{Rem_NDDE_discontinuity}
Let $x_0:[-\tau,0]\rightarrow \mathbb{R}^2$ be some initial
condition in $X$ and define $x(t)$, starting from $x_{0}$ according
to the so-called method of steps (see e.g. \cite{Smith11}); having
determined $x(t)$ on the interval $[-\tau, (k-1) \tau)$ for $k \geq
1$, we determine $x(t)$ on $[(k-1)\tau,k\tau)$ using
\eqref{Eq_excompact_form}.  Then $x(0^+)-x(0^-)=B x_0(-\tau)+f
-x_0(0)$ with $x(0^{+/-})=\lim_{t\rightarrow 0^{+/-}} x(t)$; unless
$x_{0}$ satisfies \eqref{Eq_excompact_form} at $t=0$, we obtain a
discontinuity in $x$ at $t=0$, i.e. $x(0^+)-x(0^-)\neq 0$, that will
propagate to each $t=k\tau$, $k \in \mathbb{N}$. To avoid such a
phenomenon we need to assume that $Bx_0(-\tau)+f =x_0(0)$ which
translates to $V(0,0)+R_0 I(0,0)=V_0(0)+R_0 I_0(0)=E$ and
$I(1,0)=I_0(1)=0$ in the original problem BVP
\eqref{Eq_line_transmission}-\eqref{Eq_BP1}, requiring in other
words that the initial condition $I_{0}$, $V_{0}$ must satisfy
the boundary conditions \eqref{Eq_BP1}. This phenomenon is
reminiscent of the well-known phenomena of singularity propagation
along the characteristics, typical of hyperbolic PDEs, see e.g.
\cite{ZT86}.
\end{Rmk}

Now, if we consider instead of \eqref{Eq_BP1} the following linear boundary conditions
\bea \label{Eq_BC2}
\frac{d v(0,t)}{dt}+R_0 \frac{d  i(0,t)}{dt}=E,\;\;
\frac{d i(1,t)}{dt}=0,
\eea
similar computations lead us to the following functional equation which is now an NDDE of the form
\bea \label{Eq_NFDE_Telegrapher}
\frac{d}{dt} D_0 x_t=f,
\eea
giving another formulation of the Cauchy problem associated
with the  boundary value problem (BVP) \eqref{Eq_line_transmission}-\eqref{Eq_BC2}.  Note that, here  the operator $D_0$ is still defined as $ D_0\Phi:= \Phi(0)-B\Phi(-\tau)$  acting on functions
$\Phi \in X:=C([-\tau,0],\mathbb{R}^2)$.   From this introductory example,  we can see how the difference operator
$D_0$ arises naturally in the NDDE formulation. Further details concerning the relationship
between difference equations and NDDEs, will be given in the next section.

\begin{Rmk}\label{rmk:GFDapp}
As mentioned above, the telegrapher's equations are not the
only examples of hyperbolic PDEs which can be reformulated as
systems involving retarded arguments of similar nature
to NDDEs. For instance, in modeling of the El Ni\~no-Southern
Oscillation (ENSO), the authors in \cite{GT00} have proposed such a
neutral formulation derived from a model introduced in
\cite{JN93a, JN93b}.  In these later works the  ENSO dynamics were
described by a linearized shallow-water (hyperbolic) model on an
equatorial $\beta$-plane coupled with an advection equation
describing the sea surface temperature changes at the earth's
equator. In \cite{GT00}, by dropping the spatial dependence in
the advection equation, the authors reformulated  the resulting
PDE-ODE model as a system made up of difference equations
coupled with ODEs. Note that, coupled systems of ODEs with
difference equations arise in other physical applications dealing
with wave propagations phenomena, and are strongly related to the
theory of NDDEs, see \cite{Niculescu_stab}.
\end{Rmk}

\subsubsection{Existence of invariant measures for NFDEs: preliminaries and abstract result}
\label{Sec_prelimNDDE}

 We recall here some preliminaries concerning neutral functional
differential equations used in the statement and the proof of
Theorem~\ref{Thm_existence_attractor} below concerning the existence
of global attractors and hence of invariant measures for the
semigroup generated by such equations.   In the next subsection (see
Theorem~\ref{THM_IM_NDDE} below) we will provide concrete conditions
which allow us to apply this abstract result to a broad class of
NDDEs. While, as we said above, the existence of invariant measures
for NFDEs is new, the exposition here largely follows the framework
presented in  \cite{HaleLunel93}, to which we refer the reader for
further details on the extensive general theory.

Throughout what follows we take $\langle \cdot, \cdot \rangle$ be the inner product on $\mathbb{R}^n$, with $\vert \cdot \vert$ the corresponding norm.
Let $\tau > 0,$ and take $X := C([-\tau,0],\mathbb{R}^n)$
be the space of continuous
functions taking $[-\tau,0]$ into $\mathbb{R}^n$ which will be endowed in what follows with the
topology induced by the supremum norm viz.
$$\vert \phi\vert_{\infty} :=\underset{-\tau\leq\theta \leq 0 }\sup\vert \phi(\theta)\vert, \; \mbox{ for }\phi \in X.$$
We denote by $\mathcal{M}_n(\mathbb{R})$, the space of $n\times n$ matrix  with real coefficients
which we endow with the natural norm i.e.
$\| M \|:=\sup\{ \vert M x\vert: \; x\in\mathbb{R}^n,\; \vert x\vert=1\}$ for any $M \in \mathcal{M}_n(\mathbb{R})$.

Throughout this section we consider a fixed map $M$ taking $[-\tau,0]$
into $\mathcal{M}_n(\mathbb{R})$ which is of bounded variations in the sense that,
$$
    \underset{[-\tau,0]}{\mbox{Var}(M)}:=\sup \sum_i \|M(\theta_i)-M(\theta_{i+1})\| < \infty,
$$
the supremum being taken over all finite partitions $-\tau=\theta_0\leq \theta_1\leq \cdots\leq \theta_k=0$ of the interval $[-\tau,0]$.
We assume furthermore that
$$
    \forall\; s\in (0,\tau],\; \underset{[-s,0]}{\mbox{Var}(M)}\leq \gamma(s),
$$
with $\gamma$  a continuous nondecreasing scalar function for $s\geq 0$, such that $\gamma(0)=0$. In
language of measure theory, this last requirement implies
that the vector measure $M$ (after identifying $M_n$ with $\mathbb{R}^{n^2}$) is {\it nonatomic} at zero, i.e. $M(0)-M(0^{-})=0$.
In other words $M$ does not attribute a non-zero measure to $\{0\}$.
This assumption about $M$ being {\it non-atomic} at zero gives a sufficient
condition to develop a theory of existence and uniqueness of solutions of NFDEs of the form \eqref{Eq_NDDE};
see \cite{HaleLunel93} and Proposition \ref{Prop_existenceNDDE} below.

With this $M$ in hand we define a bounded linear operator $D:X\rightarrow \mathbb{R}^n$ according to
\be\label{Eq_D_isnonatomic}
D\phi:=\phi(0)-\int_{-\tau}^0  [d M(\theta)]\phi(\theta).\;\footnote{Note that, by application of the Bartle-Dunford-Schwartz
theorem (see \cite[Theorem 3.2]{Bartle_al55}), any bounded linear
operator  $L$ from $X$ to $\mathbb{R}^n$ may be represented as  Stieltjes integral $\int_{-\tau}^0  [d N(\theta)]\phi(\theta)$ where $N$ is an $\mathbb{R}^{n^2}$ valued measure
on Borel sets of $[-\tau,0]$ which is of bounded variations. Operators having the form of  \eqref{Eq_D_isnonatomic} may therefore be interpreted as those which
have the representation corresponding to $N=\delta_0-M$ with $M$ being non-atomic at zero and $\delta_0$ being the Dirac measure at $0$.}
\ee
In \eqref{Eq_D_isnonatomic}, the notation  $d M(\theta)$ before the integrand $\phi(\theta)$ emphasizes that $M(\theta)$ is a matrix and $\phi(\theta)$ is a column vector so that the integral is column  vector-valued.

\begin{Rmk}
A classical and fundamental variety of examples of such operators $D$ arises
when we take $M$ to be a step function. This leads to
$$D\phi=\phi(0)-\sum_{k=1}^{N} B_k \phi(-\tau_k),$$ where the $B_k$'s
are $n\times n$ matrices and the delays satisfy $0<\tau_1<...<\tau_N\leq
\tau$. This type of operator $D$ is very often encountered in the applications,
see e.g. \cite{Michiels_Niculescu} and the references therein.
\end{Rmk}

Let $x$ be an $\R^n$ valued, continuous function defined on an interval $[-\tau,T]$ for some $T>0$.
For $t \in [0,T]$, we
define $x_t\in X$  as being the ``copy" of $x$ over the time
interval $[t-\tau,t]$ shifted down to $[-\tau,0]$, i.e. for each $t \in
[0,T],$ we define $x_t \in X$ via $x_t(\theta)=x(t+\theta)$ for all $\theta\in[-\tau,0]$.
For simplicity, let us assume that  $f:X\rightarrow \mathbb{R}^n$ is as regular as is
needed for a moment. A {\it neutral functional differential equation (NFDE)} is
then given by the following relation:
\be\label{Eq_NDDE}
\frac{d}{dt} Dx_t=f(x_t),
\ee
where $\frac{d}{dt}$ denotes the right hand side derivative at $t$. The
initial data is given as element in $\phi \in X.$ To make sense to
\eqref{Eq_NDDE} it is required that $t\rightarrow Dx_t $ is $C^1$ as a map with values in $\mathbb{R}^n$; cf. \cite{HaleLunel93}. More precisely we have the following definition of solutions of \eqref{Eq_NDDE}:

\begin{Def}\label{Def_solNDDE}
For a given $\phi \in X,$ we say that
$x_t(\cdot;\phi)$ is a solution of \eqref{Eq_NDDE} on the interval
$[0,\alpha_{\phi}),$ $\alpha_{\phi}>0,$ with initial data $\phi$ at
$t=0$, if $x_0(\cdot, \phi) = \phi(\cdot)$, $x_t(\cdot;\phi)\in X$, for all $t \in [0, \alpha_{\phi})$,
the map $t\rightarrow Dx_t(\cdot;\phi)$ is in $C^1([0,\alpha_{\phi}),\mathbb{R}^n)$, and $x_t(\cdot, \phi)$ satisfies \eqref{Eq_NDDE}
for all $t \in [0, \alpha_\phi)$.
\end{Def}

\begin{Rmk}\label{rem_regularity}
\mbox{}
\begin{itemize}
\item[(i)] Note that, since $M$ is assumed to be non atomic at zero, we have that
$\int_{-\tau}^0  [d M(\theta)]\phi(\theta)\xrightarrow{\tau \rightarrow 0}0$
 for any $\phi\in X$.  In this limiting case $Dx_t=x(t)$ for all $t\geq 0$, which means that the
 the case of ODEs corresponds formally to $\tau=0$ with the well-posed problem of $\dot{x}=f(x)$
 reducing from Definition~\ref{Def_solNDDE} to the classical sense of Hadamard.
\item[(ii)] Similarly to the case of difference equations, starting from $C^0$ initial data, singularities may propagate but at the level of the derivatives (cf. Remark \ref{Rem_NDDE_discontinuity}) and we do not have in general  that $t\rightarrow x_t$ is $C^1$ but in such a case,
\eqref{Eq_NDDE} reduces to $D \dot{x}_t=f(x_t),$ an equation in
which the derivative occur with delayed arguments; see \cite{HaleLunel93} for further details.
\end{itemize}
\end{Rmk}

We recall now the main results on existence, uniqueness and continuous dependence of the solutions of \eqref{Eq_NDDE}, and we refer to \cite{HaleLunel93} for a
proof of these properties.
\begin{Prop}\label{Prop_existenceNDDE}
Let $D$ be as given by \eqref{Eq_D_isnonatomic} and consider any $f$ be in $C^1(X,\mathbb{R}^n)$.
Then, for any $\phi \in X$, there exists a unique solution $x_t(\cdot;\phi)$ of \eqref{Eq_NDDE}
through $\phi$ at $t=0$ defined on a maximal interval
$[0,\alpha_{\phi})$, such that if $\alpha_{\phi} <\infty$ then
$\underset{t\rightarrow \alpha_{\phi}^{-}}\lim
|x_t(\cdot;\phi)|_{\infty}=\infty$. Furthermore, the map $(t,\phi) \rightarrow x_t(\cdot;\phi)$ is continuous
from the set $[0,\alpha_{\phi})\times X$ into $X$.
If we assume that all solutions  of \eqref{Eq_NDDE}
are global in time, i.e. we assume that $\alpha_\phi = \infty$ for all $\phi \in X$, and we introduce
$S_{D,f}(t)\phi:=x_t(\cdot,\phi),$ then $\{S_{D,f}(t)\}_{t\geq 0}$ is a
continuous semigroup\footnote{In the sense recalled at the beginning of Section \ref{sec:MainResults}.} acting on $X$.
\end{Prop}
\noindent Hereafter in this section, we will assume that $D$  as given by \eqref{Eq_D_isnonatomic} and $f\in C^1(X,\mathbb{R}^n)$ are such that all solutions of \eqref{Eq_NDDE} are global in
time.\footnote{See Theorem 2.1 in \cite{Cruz_Hale69} and references therein, for classical conditions ensuring existence of global solutions in time. See also Lemma \ref{Lem_dissip} proved below.}

We next describe a  decomposition of the semigroup $S_{D,f}(t)$,
given in precise form in Lemma~\ref{Lem_decomp_NDDE} below,
which splits $S_{D,f}(t)$ into two semigroups,
one of which is compact and the other of which decays exponentially towards zero.
As we shall see, this decomposition will depend on the asymptotic stability of the zero solution
of a certain difference equation associated with $D$; cf. \eqref{Eq_newD},  \eqref{Eq_diff_eq}
below.
Note that, in view of Proposition~\ref{thm:AttractorCriteria}, Lemma~\ref{Lem_decomp_NDDE}
will provide a sufficient condition for the existence of a global attractor.

To formulate this result we will assume that the operator $D$ is given by,
\bea\label{Eq_newD}
D\phi &= D_0 \phi -\int_{-\tau}^0 M(\theta)\phi(\theta)d\theta, \;  \mbox{ with,}\\
D_0\phi &= \phi(0)-\sum_{k=1}^{N} B_k \phi(-\tau_k) \mbox{ for } \phi \in X,
\eea
where the
$B_k$'s are $n\times n$ matrices with the delays of the form
$0<\tau_1<\cdot \cdot \cdot<\tau_N\leq \tau$, and where the family of
matrices $\{M(\theta)\}_{\theta\in [-\tau,0]}$ satisfies $\int_{-\tau}^0
\| M(\theta)\|d\theta < \infty$.
We now consider a linear difference equation associated with the operator $D_{0}$ given by:
\be\label{Eq_diff_eq}
D_0 y_t= y(t)-\sum_{k=1}^{N} B_k y(t-\tau_k)=0, \mbox{ for } t\geq 0.
\ee
Clearly, any initial value $y_0$ of this equation must live in the null space $\mathcal{N}(D_0),$ so we can restrict
our attention to $X_{D_{0}}$ the closed subset of $X$ defined according to:
$$
X_{D_{0}} := X\cap \mathcal{N}(D_0)=\{\phi \in X\;:\; D_0 \phi=0\}.
$$
Then, for any $\phi \in X_{D_0},$ it can be shown that there is a
unique solution $t\rightarrow y_t(\cdot;\phi)$ of \eqref{Eq_diff_eq}
which is continuous from $[-\tau,\infty)$ into $X_{D_0}$ and
coincides with $\phi$ on $[-\tau,0]$; see \cite{HaleLunel93}.\footnote{Note that such a $y_{t}$ can be derived from $\phi
\in X_{D_{0}}$ using the method of steps;  cf.
Remark~\ref{Rem_NDDE_discontinuity}, above.} This implies that the
translation along the solutions of \eqref{Eq_diff_eq}, $$
S_{D_0}(t)\phi=y_t(\cdot;\phi)$$ defines a $C^0$ semigroup on
$X_{D_0}$. As shown in  \cite[Theorem 3.3. p. 284]{Hale77}, the
infinitesimal generator $A_0$ of this semigroup is given by
$$
A_0\phi=\dot{\phi} \mbox{ for }  \phi\in \mathcal{D}(A_0)=\{\phi \in X_{D_0}:\; \dot{\phi}\in X_{D_0}\},
$$
and the spectrum of $A_0,$ $\sigma(A_0),$ is then given by
$$
\sigma(A_0)=\{\lambda \in \mathbb{C}\;:\; \mbox{det}(H(\lambda))=0\}
$$
where,
$$
H(\lambda)=I-\sum_{k=1}^{N} B_k e^{-\lambda \tau_k}.
$$
This last relationship can be heuristically understood by looking for nontrivial solutions of  \eqref{Eq_diff_eq} of the form $y_t(\theta)=e^{\lambda (t+\theta)}v_{\lambda},$ $\theta\in [-\tau,0]$ where $v_{\lambda}$ is
some vector living in $\mbox{ker}(H(\lambda))$.

As we already mentioned, Lemma \ref{Lem_decomp_NDDE}
below is conditioned  to the stability of the  zero solution of
\eqref{Eq_diff_eq}.  This stability is related to the location of
the spectrum of $A_0$ with respect to the imaginary axis, which is
itself determined by the roots of the characteristic equation
$\mbox{det}(H(\lambda))=0$; see e.g. \cite{Michiels_Niculescu}.  According to
\cite[Theorem 4.1,  p. 287, Lemma 3.3, p. 284]{Hale77}, we deduce
that if the real part of the rightmost eigenvalue of $A_0$ is
located in the left half-complex plane, i.e.  if
\be
\label{eq_stabilityofAo}
\mathfrak{R}_{A_0}:=\sup\left\{\mbox{Re} (\lambda)\;:\;
\mbox{det}\left[I-\sum_{k=1}^{N} B_k e^{-\lambda \tau_k} \right]=0\right\}
<0,
\ee
then $0$ is exponentially stable in $X_{D_0}$ for
$S_{D_0}(t)$  in the sense that there exist $C>0$  such that for all
$t\geq 0$ and all $\phi\in X_{D_0}$ we have
$|S_{D_0}(t)\phi|_{\infty}\leq C e^{-\alpha t} |\phi|_{\infty}$,
with $\alpha=\mathfrak{R}_{A_0}$. In what follows we will summarize this property by just saying that $D_0$
is {\it stable}.

 With this background now in place, the desired decomposition lemma may now be stated as follows:
\begin{Lem}\label{Lem_decomp_NDDE}
Consider $D$,  and $f$ such that \eqref{Eq_NDDE} is
globally well posed in the sense of Proposition~\ref{Prop_existenceNDDE}. Let
$\{S_{D,f}(t)\}_{t\geq 0}$ be the continuous semigroup generated by
\eqref{Eq_NDDE} with $D$ satisfying \eqref{Eq_newD} such that the
 difference operator $D_0$ is stable in the sense that $\mathfrak{R}_{A_{0}}$ as given by  \eqref{eq_stabilityofAo} is strictly less than zero.
Then there exists a (time-independent) linear bounded operator
$\Psi:X \rightarrow X_{D_0}$ such that,
\be\label{Eq_decomp_semigroup} S_{D,f}(t)=S_{D_0}(t) \circ \Psi +
U_{D}(t), \quad t\geq 0, \ee where $\{U_D(t)\}_{t\geq 0}$ is a
compact semigroup on $X$, and for every $K>0$,
$$\underset{|\phi|_{\infty}\leq K}\sup |S_{D_0}(t) \circ \Psi
(\phi)|_{\infty}\xrightarrow{t\rightarrow \infty}0 \textrm{
exponentially with uniform decay rate } \mathfrak{R}_{A_0}.$$
\end{Lem}

 Proofs of this lemma  may be found in \cite{Cruz_Hale69,
Cruz_Hale70}; cf. also \cite{HaleLunel93} for a unified treatment.
Note that the noncompact part of $S_{D,f}(t)$ comes from the
semigroup $S_{D_0}(t)$ which is associated with the difference
equation \eqref{Eq_diff_eq} for which $y_t\equiv 0$ is solution. The
linear bounded operator $\Psi$ arises essentially in order to map
$X$  into $X_{D_0}$, the domain of $S_{D_0}(t)$.  Consider, for
instance, the case $M=0$, and $B_k=0$ for all $k\in \{1,...,N\}$ in
\eqref{Eq_newD}, i.e. when the NFDE \eqref{Eq_NDDE} becomes a
standard RFDE with no retarded arguments on the derivative.  Here,
$\Psi$ is simply defined as the shift $\Psi(\phi)=\phi-\phi(0)$ so
that $\Psi (\phi)\in X_{D_0}$ where, in this case, $X_{D_0}=\{\phi
\in X:\; \phi(0)=0\}$. In the general case,
$\Psi=\mbox{Id}_{X}-\Phi\circ D_0$ where
$\Phi:\mathbb{R}^n\rightarrow X_{D_0}$ is the right inverse of
$D_0$, i.e. $D_0\circ \Phi=\mbox{Id}_{\mathbb{R}^n}$. Indeed, for
such a $\Psi$, we get $\Psi(\phi)=\phi-\Phi\circ D_{0} (\phi)$ for
all $\phi \in X$ which gives $0$ when we compose to the left by
$D_0$ so that, as desired, $\Psi(\phi)\in X_{D_0}$ for all $\phi \in
X$. The existence of such a right inverse $\Phi$ is established in
\cite{Hale77} for operators $D$ of the form \eqref{Eq_D_isnonatomic}
with $M$ nonatomic and in particular for operators such as $D_0$
given in \eqref{Eq_newD}.

Combining Lemma \ref{Lem_decomp_NDDE} and Proposition \ref{thm:AttractorCriteria},
we conclude the following general result concerning the existence of a global attractor
for a wide class of NFDEs. The existence of invariant probability measures, associated
with temporal averages via generalized limits,
then follow for such NFDEs in view of Theorem \ref{thm:Prop1} and Theorem \ref{thm:Prop2}.

\begin{Thm}\label{Thm_existence_attractor}
 Let $X := C([-\tau,0],
\mathbb{R}^n)$.  Assume that $D$ satisfies \eqref{Eq_newD} with $D_0$ stable
in the sense that $\mathfrak{R}_{A_0} <0$,
where $\mathfrak{R}_{A_{0}}$ is defined according to \eqref{eq_stabilityofAo}.
Suppose that for this $D$ and some given $f\in C^1(X,\mathbb{R}^n)$ that
all of the solutions of the associated system \eqref{Eq_NDDE}
are well-defined for all $t\geq 0$ in the sense of Definition~\ref{Def_solNDDE}.
Let $\{S_{D,f}(t)\}_{t\geq
0}$ be the continuous semigroup generated by \eqref{Eq_NDDE} on $X$.
If there exists an absorbing set
$\mathfrak{B}\subset X$, for $S_{D,f}$ then the omega-limit set,
$\omega(\mathfrak{B})$, is the global attractor of $S_{D,f}$.
Moreover, for any given generalized Banach Limit $\LIM$, and for any $\brlMsr_0 \in Pr(X)$ there exists an invariant
    measure $\brlMsr \in Pr(X)$ for $\{ S_{D,f}(t)\}_{t \geq 0}$ whose support is contained in $\mathcal{A}$
    and such that
    \be
    \label{eq:WeakErgodicAve_NDDE}
       \int_{X} \varphi(u) \dr \brlMsr(u)
       = \LIM \frac{1}{T}\int_0^T \int_{X} \varphi(S_{D,f}(t) u) \dr \brlMsr_{0}(u) \dr t,
       \; \textrm{ for any } \varphi \in C_b(X).
  \ee
Furthermore,  if $\mathfrak{m}_0=\delta_{\phi}$, the Dirac measures for some $\phi \in X$, then
\eqref{eq:WeakErgodicAve_NDDE} holds for any $\varphi \in C(X)$.
\end{Thm}
\begin{proof}
The result follow as an obvious consequence of Lemma \ref{Lem_decomp_NDDE},  Proposition \ref{thm:AttractorCriteria}, Theorem \ref{thm:Prop1} and Theorem \ref{thm:Prop2}.
\end{proof}

\subsubsection{Invariant measures for NDDEs}\label{Sec_resultsNDDE}

In view of its generality, Theorem~\ref{Thm_existence_attractor}
provides a powerful tool that may be employed to establish
the existence of a global attractor and hence of invariant probability measures for particular NDDEs.

Of course, in practice, we need to verify the
condition $\mathfrak{R}_{A_0}<0$, which ensures the splitting of the semigroup generated by the NDDE of interest,
and furthermore to establish the existence of a bounded absorbing
set. We next provide below a useful criterium for the verification of this
latter dissipativity condition.  This criteria is satisfied for a special but
still rather broad class of NDDEs. The
verification of $\mathfrak{R}_{A_0}<0$ has been and is still an
intensive topic of research. Many criterion exist depending on the
situation of interest; see e.g. \cite{Michiels_Niculescu}. In what follows we
will need only the Schur-Cohn criterion, see e.g.
\cite{Niculescu_stab}.

\begin{Lem}\label{Lem_dissip}
Let $\tau$ be a positive constant, and $B$ a $n\times n$ matrix with real entries. Let $g$ be in $C^1(\mathbb{R}^n\times \mathbb{R}^n;\mathbb{R}^n)$.
Consider the neutral delay differential equation given by
\be
\label{Eq_NDDE_dissip}
\frac{d}{dt}\big(x(t)-B
x(t-\tau)\big)=g(x(t),x(t-\tau)),  \mbox{ for } t\geq 0,
\ee
on the
phase space $X:=C([-\tau,0], \mathbb{R}^n)$.

Assume that  there exist $\alpha, \gamma >0$, $\beta \in \mathbb{R}$
such that \be \label{Dissip_g} \langle u-Bv, g(u,v)  \rangle \leq
\gamma -\alpha \vert u\vert^2 +\beta \vert v\vert^2, \;\forall\; u,v
\in \mathbb{R}^n, \ee and that,

\be \label{Dissip_cond} \mathfrak{C}:=\|B\| + \sqrt{(1+\|B\|)^2
e^{-\alpha \tau} + 2(\beta +\alpha \|B\|^2) \frac{1-e^{-\alpha
\tau}}{\alpha}}  <1. \ee Then the solutions of
\eqref{Eq_NDDE_dissip} are global in time, and the ball $B(0, 2  r
\sqrt{\sum_{k=0}^{\infty} \mathfrak{C}^k})$ in $X=C([-\tau,0],
\mathbb{R}^n)$  with $r=\sqrt{\frac{2 \gamma
(1-e^{-\alpha\tau})}{\alpha}}$,  is absorbing for the continuous
semigroup generated by \eqref{Eq_NDDE_dissip} on $C([-\tau,0],
\mathbb{R}^n)$.
\end{Lem}

\begin{proof}
Equation \eqref{Eq_NDDE_dissip} may be written in the functional
form $\frac{d}{dt}D_0 x_t= f(x_t)$ with $D_0\phi :=
\phi(0)-B\phi(-\tau)$ and $f(\phi):=g(\phi(0),\phi(-\tau))$ for any
$\phi\in X,$ and since $g\in C^1(\mathbb{R}^n\times
\mathbb{R}^n;\mathbb{R}^n),$ we have that $f \in
C^1(X,\mathbb{R}^n)$. The
local existence and uniqueness of solutions up to a maximal
time $\alpha_\phi$
\eqref{Eq_NDDE_dissip} is ensured by Proposition
\ref{Prop_existenceNDDE}. From the same proposition, in order to
conclude to the global existence of solutions of \eqref{Eq_NDDE_dissip},
i.e. to show  that $\alpha_\phi = \infty$ for every $\phi \in X$,
it suffices to show that for any
$\phi$ and for any $T >0$,
$\sup_{t \in [0, \min\{T, \alpha_\phi\})} |x_t(\cdot;\varphi)|_{\infty} < \infty$.
The estimates provided below for the existence of an
absorbing ball contain such an argument implicitly
and existence of a continuous semigroup associated to
\eqref{Eq_NDDE_dissip}
follows from
Proposition \ref{Prop_existenceNDDE}.
We turn next to these estimates.

We first show that, for any $0 \leq t_1 <t_2$,
\bea \label{Ineq_gronwall}
\vert x(t_2) -Bx(t_2-\tau)\vert^2  \leq & e^{-\alpha (t_2-t_1)} \vert x(t_1) -Bx(t_1-\tau)\vert^2 +2\frac{\gamma}{\alpha} (1-e^{-\alpha (t_2-t_1)})+\\
& 2 (\beta +  \alpha \|B\|^2) \underset{s \in [t_1-\tau, t_2-\tau]}{\sup} \vert x(s) \vert^2 (1-e^{-\alpha (t_2-t_1)})/\alpha.
\eea
This inequality may be derived by considering the quantity $u(t):=\vert x(t) -Bx(t-\tau)\vert^2=\langle x(t) -Bx(t-\tau), x(t) -Bx(t-\tau) \rangle$.  Since
$D_0 x \in C^1([0, \infty), \R^n)$, $u$ is differentiable we find with \eqref{Dissip_cond} that
$$ u'(t) \leq 2 (\gamma -\alpha \vert x(t) \vert^2 +\beta \vert x(t - \tau)\vert^2),$$
Now since $2 \vert x(t) \vert^2 \geq u(t)-2 \vert B x(t-\tau)\vert^2$, we infer
$$ u'(t)\leq 2 \gamma  -\alpha u(t) + 2(\beta +  \alpha \|B\|^2) \vert  x(t-\tau)\vert^2, \quad \mbox{ for } t>0.$$
The inequality \eqref{Ineq_gronwall}  is then easily derived by multiplying this last inequality by $e^{\alpha t}$ and integrating  between $t_1$ and $t_2$.

Let us take now $t$ to be in $(0,\tau].$ We use  \eqref{Ineq_gronwall} with $t_1=0$ and $t_2=t,$ it follows that,
$$
 \vert x(t) -Bx(t-\tau)\vert^2 \leq e^{-\alpha t} \vert x(0) -Bx(-\tau)\vert^2 +2\frac{\gamma}{\alpha} (1-e^{-\alpha t})+
 2 (\beta +  \alpha \|B\|^2) \vert \phi \vert_{\infty}^2 \frac{(1-e^{-\alpha t})}{\alpha},
$$
and since $\vert x(0) -Bx(-\tau)\vert^2 \leq (1+\|B\|)^2 \vert \phi \vert_{\infty}^2,$ we get
$$
\vert x(t) -Bx(t-\tau)\vert^2  \leq \Big[(1+\|B\|)^2+2 (\beta +  \alpha \|B\|^2)  \frac{(1-e^{-\alpha \tau })}{\alpha}\Big]\vert \phi \vert_{\infty}^2
+2\frac{\gamma}{\alpha} (1-e^{-\alpha \tau}),
$$
which gives finally,
\be\label{eq:basecaseIndAbs}
\vert x(t) \vert \leq \mathfrak{C}_0 \vert \phi \vert_{\infty} + r, \quad \mbox{ for } t\in (0,\tau],
\ee
with $\mathfrak{C}_0:=\sqrt{(1+\|B\|)^2+2 (\beta +  \alpha \|B\|^2)  \frac{(1-e^{-\alpha \tau })} {\alpha}}+\|B\|,$ and $r=\sqrt{2\frac{\gamma}{\alpha} (1-e^{-\alpha \tau})}.$
For $t>\tau$ we take $t_1=t-\tau$ and $t_2=t$ in \eqref{Ineq_gronwall}.  By arguing in similar manner to the previous case we obtain:
\bea\label{eq:indStepIndAbs}
\vert x(t) -Bx(t-\tau)\vert^2  \leq \Big[(1+\|B\|)^2e^{-\alpha \tau}+2 (\beta +  \alpha \|B\|^2)  \frac{(1-e^{-\alpha \tau })}{\alpha}\Big]\vert x_{t-\tau} \vert_{\infty}^2 +r^2,
\eea
which leads to,
\be
 \vert x(t) \vert \leq \mathfrak{C} \vert x_{t-\tau} \vert_{\infty} +r, \quad \mbox{ for } t> \tau,
 \ee
 with $\mathfrak{C}$ given in \eqref{Dissip_cond}. We infer from \eqref{eq:basecaseIndAbs}, \eqref{eq:indStepIndAbs} and a simple induction that,
 for every $k \geq 0$,
  \begin{equation}\label{eq:SimpleDispInduct}
   \vert x(t) \vert \leq \mathfrak{C}^k \mathfrak{C_0} \vert \phi \vert_{\infty} +r \sum_{j=0}^{k} \mathfrak{C}^j, \; \mbox{ for } t\in(k\tau,(k+1)\tau].
 \end{equation}
 The existence of the bounded absorbing set for the semigroup generated by \eqref{Eq_NDDE_dissip} therefore follows, completing the proof.
 \end{proof}

Note that the dissipation condition \eqref{Dissip_cond}  is
valid for ``large'' delays provided that $2\beta<{\alpha}$ and
$\|B\|$ is appropriately chosen with respect to
$\frac{\beta}{\alpha}$.   More precisely we have
the following:

\begin{Lem}\label{Lem_analysisC}
(Dissipation for large delays) Suppose that $2\beta<{\alpha}$ and  $\| B\|<
\sqrt{2(1-\frac{\beta}{\alpha})} -1$. Then \eqref{Dissip_cond} holds
for all $\tau > \tau^*$, with
$\tau^{*}=-\frac{1}{\alpha}\log \left(\frac{2(1-\frac{\beta}{\alpha}) - (\| B\|+1)^{2} }{2(1-\frac{\beta}{\alpha}) - (\| B\|-1)^{2}}\right)$.
\end{Lem}

\begin{proof}
Direct manipulations show that \eqref{Dissip_cond} is equivalent to:
\bea\label{Eq_dissip_poly} 
P(\| B\|) e^{-\alpha \tau} < P( \| B \| +
2), 
\eea 
where $P$ is the polynomial function given by
$P(x)=-[(x-1)^2 +2(\frac{\beta}{\alpha}-1)]$. 
Under the assumed conditions
on $\beta$ and $\alpha$,
$$
\sqrt{2\left(1-\frac{\beta}{\alpha}\right)}> 1
$$
and $P(x+2)$ possesses two real distinct roots, namely $x_1 = -1 -  \sqrt{2(1-\frac{\beta}{\alpha})}< 0$, $x_2 = -1 + \sqrt{2(1-\frac{\beta}{\alpha})}> 0$
and reaches its maximum at $-1$. As such, for all $x \in \left[0, \sqrt{2(1-\frac{\beta}{\alpha})} -1\right)$,
it is clear that $P(x), P(x+2) > 0$ and also that $P(x) -P(x+2)$ is strictly increasing and non-negative.
Since, by assumption, $\|B\|$ falls in this range, \eqref{Eq_dissip_poly} holds for all $\tau > \tau^{*}$, completing the proof.
\end{proof}

 We are now in position to prove our main theorem about the existence of invariant probability measures for
 NDDEs.

 \begin{Thm}\label{THM_IM_NDDE}
 Let $\tau$ be a positive constant, and $B$ a $n\times n$ matrix with real entries and let $g$ be in $C^1(\mathbb{R}^n\times \mathbb{R}^n;\mathbb{R}^n)$. Consider the neutral delay differential equation given by
\be \label{Eq_NDDE_dissip_bis} \frac{d}{dt}\big(x(t)-B
x(t-\tau)\big)=g(x(t),x(t-\tau)),  \mbox{ for } t\geq 0, \ee  on the
phase space $X:=C([-\tau,0], \mathbb{R}^n).$
Assume that the dissipation conditions  \eqref{Dissip_g} and
\eqref{Dissip_cond} are satisfied. Then, the continuous semigroup $\{ S(t)\}_{t\geq 0}$
acting on $X$ which is generated by \eqref{Eq_NDDE_dissip} possesses a global attractor $\mathcal{A}$
and the results in Theorems \ref{thm:Prop1}, \ref{thm:Prop2}
hold for $\{S(t)\}_{t \geq 0}$.
\end{Thm}

 \begin{proof}
Keeping the previous notations,
 we note that for all $\phi \in X,$ $D_0\phi=\phi(0)-B\phi(-\tau)$ in this case.  By Lemma \ref{Lem_dissip},
 we infer the existence of an absorbing set for $\{S(t)\}_{t \geq 0}$, and since \eqref{Eq_NDDE_dissip} is
 a particular form of the type of NFDEs handled by Lemma \ref{Lem_decomp_NDDE},
 to apply Theorem~\ref{Thm_existence_attractor} and hence to infer all of the desired
 results we need now only check that  $\mathfrak{R}_{A_0} <0,$ where $\mathfrak{R}_{A_0}$
 reduces simply here to $\sup\big\{\mbox{Re} (\lambda)\;:\;
\mbox{det}\Big[I-B  e^{-\lambda \tau} \Big]=0\big\}$. It is well
known that in such a case $D_0$ is stable (independently of the delay
$\tau$) if and only if  $\rho(B)<1$  where $\rho(B)$ is the spectral
radius of $B$.
This is the so-called Schur-Cohn condition; see
\cite{Niculescu_stab}.  This last condition is guaranteed by the imposed
dissipation condition \eqref{Dissip_cond} since $\|B\|<1$ necessarily, and
$\rho(B)\leq\|B\|$ trivially. The proof is therefore complete.
\end{proof}

\begin{Rmk}
Note that this theorem may be extended to NDDEs of the form $\frac{d}{dt}D_0x_t=G(x_t)$ with $D_0$ as given in \eqref{Eq_newD} and $G$ being the functional representation associated to $g$ in \eqref{Eq_NDDE_dissip}. The  dissipation estimates and the stability criteria are however more involved. We leave such a generalization of Theorem \ref{THM_IM_NDDE}
to the interested reader.
\end{Rmk}

\subsubsection{Application to a nonlinear Brayton-Miranker-like model}\label{Sec_resultsNDDE}

We now return to a particular system of NDDEs, \eqref{Eq_wavesfunction3}, arising from the transmission line problem discussed above
in Section~\ref{Sec_NDDE_HyperPDEs}. 
Our goal is to provide conditions under which  \eqref{Eq_wavesfunction3} exhibits a global attractor 
and hence to infer the existence of invariant measures for this example.  
In view of Theorem \ref{THM_IM_NDDE} we are left with the verification of the dissipation conditions  \eqref{Dissip_g} and
\eqref{Dissip_cond} for this system which we establish below.

Let $0<m,q<1;$ $p\in \mathbb{R}$, and let $b$ and $c$ be strictly positive real numbers. 
Let us introduce the notations $\Phi(t):=(\phi_1(t),\phi_2(t))^T, $ and ${\bf v}=(v_1,v_2)$  and ${\bf u}=(u_1,u_2).$ 
We use here bold typeface to distinguish vectors in $\mathbb{R}^2$
from real numbers. We warn the reader that the symbol $\vert \cdot\vert$ will be used both for the absolute
value of real numbers and for the Euclidean vector norm in $\R^{2}$.

Let us consider  $F_1,\; F_2:\mathbb{R}^2\times
\mathbb{R}^2\rightarrow \mathbb{R}$ two $C^1$ functions which
satisfy respectively that there exist $\alpha'>0$, $\gamma_1\geq 0$,
$M'_1\geq 0$ and $M_1>0$ such that   for all $({\bf u},{\bf
v})\in \mathbb{R}^2\times \mathbb{R}^2$,
\begin{equation}\label{Eq_NDDE_NonLin}
  \begin{split}
 u_1F_1({\bf u},{\bf v}) &\leq -\alpha'u_1^2+\gamma_1,\\
\vert v_2F_1({\bf u},{\bf v}) \vert &\leq M_1 u_1+M_1'v_2,
\end{split}
\end{equation}
and that there exist $\gamma_2\geq 0$, $M'_2\geq 0$ and $M_2>0$ such
that for all $({\bf u},{\bf v})\in \mathbb{R}^2\times
\mathbb{R}^2$,
\begin{equation}\label{Eq_NDDE_NonLin2}
  \begin{split}
u_2F_2({\bf u},{\bf v}) &\leq -\alpha' u_2^2+\gamma_2,\\
\vert v_1F_2({\bf u},{\bf v}) \vert &\leq M_2 u_2+M_2'v_1.
\end{split}
\end{equation}
For instance $F_i({\bf u},{\bf v})=-\alpha_i u_i(1+\vert {\bf v}\vert^2)^{-1}$ verifies the conditions \eqref{Eq_NDDE_NonLin} for $i=1$ and \eqref{Eq_NDDE_NonLin2} for $i=2$.

We consider now the following system of NDDEs:

 \bea \label{Eq_wavesfunction3}
\frac{d}{dt}\Big(\begin{bmatrix} \phi_1(t) \\ \phi_2(t) \end{bmatrix}-\begin{bmatrix} 0 & q\\ m & 0 \end{bmatrix} \begin{bmatrix} \phi_1(t-\tau) \\ \phi_2(t-\tau) \end{bmatrix}\Big)=\begin{bmatrix} p \\ 0 \end{bmatrix} -\begin{bmatrix} b\phi_1(t)\\ c\phi_2(t)\end{bmatrix}+\begin{bmatrix} F_1(\Phi(t),\Phi(t-\tau))\\F_2(\Phi(t),\Phi(t-\tau)) \end{bmatrix},\\
 \eea

 We introduce furthermore $k:=\max(m,q),$ which is the norm of the matrix arising in the LHS of $\eqref{Eq_wavesfunction3}$ which has to be strictly less than 1 
in view of  Lemma~\ref{Lem_analysisC}, which explains the constraints imposed on $m$ and $q$.

The condition  \eqref{Dissip_g} to satisfy here can be written as:

$$ N(u,v):=(u_1-qv_2) \Big( p-bu_1+ F_1({\bf u},{\bf v})\Big)+(u_2-mv_1)\Big(-cu_2+  F_2({\bf u},{\bf v})\Big)\leq \gamma -\alpha\vert {\bf u}\vert^2 +\beta \vert {\bf v}\vert^2,$$
with $\alpha>0, \beta\in\mathbb{R}$ and $\gamma>0$ to find independently of ${\bf u},{\bf v}
\in \mathbb{R}^2$.

An easy computation shows that,
$$ N(u,v)\leq \gamma_1 +\gamma_2 + pu_1+\vert pqv_2 \vert -bu_1^2-c u_2^2 -\alpha'  \vert {\bf u} \vert^2 +M_1 u_1+M_1'v_1+ M_2 u_2+M_2'v_2 +cm v_1 u_2 +bq u_1 v_2,$$
now by applying the Young inequality on the two last terms and  the $\epsilon$-Young inequality appropriately on the rest of the terms concerned, it is easy to show that for any $\epsilon>0$ there exists $\gamma_{\epsilon}>0$ such that,
$$ N(u,v)\leq  \gamma_{\epsilon}-\big( \frac{1}{2}\min(b,c) +\alpha'-\epsilon\big)\vert u\vert ^2 +\big(\frac{1}{2}\max(b,c)+\epsilon\big)\vert v\vert ^2,$$
which shows that \eqref{Dissip_g} is satisfied with $B=\begin{bmatrix} 0 & q\\ m & 0 \end{bmatrix}$ and $g({\bf u},{\bf v})=(p-bu_1+ F_1({\bf u},{\bf v}), -cu_2+  F_2({\bf u},{\bf v}))^T.$

Let us introduce $\alpha_{\epsilon}:=\frac{1}{2}\min(b,c)+\alpha'-\epsilon$  and $\beta_{\epsilon}:=\frac{1}{2}\max(b,c)+\epsilon.$
 Assume now that  $\alpha'$ is such that $\max(b,c) <(\frac{1}{2}\min(b,c)+\alpha')$. Let $\epsilon>0$ be fixed sufficiently small such that $2 \beta_{\epsilon} < \alpha_{\epsilon}$.
Now if we assume furthermore that
$k<-1+\sqrt{2(1-\frac{\beta_{\epsilon}}{\alpha_{\epsilon}})}$, then
from Lemma \ref{Lem_analysisC} we deduce that there exists $\tau^*$,
such that  \eqref{Dissip_cond} is satisfied for $\tau>
\tau^*$.

By using now Theorem \ref{THM_IM_NDDE}  we have thus proved the following proposition which constitutes,
to the best of the authors' knowledge, the first result regarding existence of invariant measures for
Brayton-Miranker-like models such as \eqref{Eq_wavesfunction3}.

\begin{Prop}
Assume that $\alpha'$ arising in \eqref{Eq_NDDE_NonLin} and \eqref{Eq_NDDE_NonLin2} is such that $\max(b,c) <(\frac{1}{2}\min(b,c)+\alpha')$.
For some $\epsilon>0$ such that  $2 \beta_{\epsilon} < \alpha_{\epsilon},$ assume furthermore that  the coefficients $q$ and $m$ of the matrix arising in \eqref{Eq_wavesfunction3} satisfy:
$$
\max(q,m)<-1+\sqrt{2\left(1-\frac{\beta_{\epsilon}}{\alpha_{\epsilon}}\right)}
$$
where $\alpha_{\epsilon}:=\frac{1}{2}\min(b,c)+\alpha'-\epsilon$ and
$\beta_{\epsilon}:=\frac{1}{2}\max(b,c)+\epsilon$. Let
$\tau^{*}=-\frac{1}{\alpha_{\epsilon}}\log\big(P(\max(q,m)+2)/P(\max(q,m))\big)$
with $P(x):=-[(x-1)^2
+2(\frac{\beta_{\epsilon}}{\alpha_{\epsilon}}-1)]$. Then for all
$\tau >\tau^*$ the continuous semigroup generated by
\eqref{Eq_wavesfunction3} possesses a global attractor $\mathcal{A}$
living in $C([-\tau,0],\mathbb{R}^2)$. This semigroup possesses
furthermore invariant probability measures whose support is
contained in $\mathcal{A}$ and which satisfy the ``weak
ergodic" property \eqref{eq:WeakErgodicAve_NDDE}.
 \end{Prop}

\section*{Appendix: invariant measures and the global attractor}

For the sake of completeness we recall briefly a proof of the classical fact
that invariant probability measures must always have their support contained
in the global attractor;  see e.g. \cite[Lemma 4.2]{BCFM95}
or \cite{FoiasManleyRosaTemam1} given in the concrete case of the 2D Navier-Stokes
equations.
\begin{Lem}\label{thm:InvarOnAttr}
  Let $\{S(t)\}_{t \geq 0}$ be a continuous semigroup defined
  on a metric space $(X,d)$.  Suppose that $\{S(t)\}_{t \geq 0}$
  possesses a global attractor $\mathcal{A}$.  Then any invariant
  borel probability measure $\brlMsr $ (relative to $\{S(t)\}_{t \geq 0}$) has
  its support contained in $\mathcal{A}$ so that, in particular,
  $\brlMsr (\mathcal{A}) =1$.
\end{Lem}
\begin{proof}
    For $\delta > 0$ form the sets $\mathcal{A}_{\delta} = \{y : \inf_{x \in \mathcal{A}}d(x,y) < \delta \}$.
    In view of the basic continuity properties of measures, it is sufficient to show that,
        $\brlMsr (\mathcal{A}_{\delta}) =1, \textrm{ for every } \delta > 0$, since, evidently, $\mathcal{A} = \cap_{\delta > 0} \mathcal{A}_{\delta}$.

    To this end fix $\delta > 0$.  Since $\mathcal{A}$ is attracting, for every $R > 0$ and each $x \in X$ we may select $t_{R} > 0$ such that
    $
       S(t_{R})B_{R}(x) \subset \mathcal{A}_{\delta}
    $, where $B_{R}(x)$ is the ball of radius $R$ around the point $x$.
    This implies,
    $
        B_{R}(x) \subset S(t_{R})^{-1}S(t_{R})B_{R}(x) \subset S(t_{R})^{-1}\mathcal{A}_{\delta}
    $.
    Thus, due to the invariance of $\brlMsr $, we have, for every $R > 0$ that
    $
      \brlMsr (B_{R}(x)) \leq \brlMsr ( S(t_{R})^{-1}\mathcal{A}_{\delta}) = \brlMsr (\mathcal{A}_{\delta}).
    $
    Since for every $x$ the collection $\{B_{R}(x)\}_{R > 0}$ is a nested collection of sets with $\cup_{R > 0} B_{R}(x) = X$ we now
    infer $\brlMsr (\mathcal{A}_{\delta}) =1$ by again invoking basic continuity properties of measures.  The proof is therefore complete.
    \end{proof}

\section*{Acknowledgments}
This article, in its several incarnations, has benefitted from the extensive feedback of
Roger Temam.  We would also like to thank 
Silviu Niculescu, Vittorino Pata, and Francesco Di Plinio for numerous
insightful discussions concerning, respectively, neutral delay differential equations (SN)
and the memory framework of Dafermos (VP \& FDP).  Finally, we would like to acknowledge the 
anonymous referee for his or her many helpful comments.
This work was partially supported by the National Science Foundation
under the grants NSF-DMS-1004638, NSF-DMS-0906440 (NGH), NSF-DMS-1049253 (MDC), by
the US Department of Energy grant DE-FG02-07ER64439 (MDC) and by the
Research Fund of Indiana University (NGH).

\footnotesize
\bibliographystyle{amsalpha}
\bibliography{ref_negh,ref_mk}

\newcommand{\etalchar}[1]{$^{#1}$}
\providecommand{\bysame}{\leavevmode\hbox to3em{\hrulefill}\thinspace}
\providecommand{\MR}{\relax\ifhmode\unskip\space\fi MR }
\providecommand{\MRhref}[2]{%
  \href{http://www.ams.org/mathscinet-getitem?mr=#1}{#2}
}
\providecommand{\href}[2]{#2}
\begin{thebibliography}{CDPGHP11}

\bibitem[AM60]{Abolina_Myskis60}
V.~{{\`E}}. Abolinya and A.~D. My{\v{s}}kis, \emph{Mixed problems for
  quasi-linear hyperbolic systems in the plane}, Mat. Sb. (N.S.) \textbf{50
  (92)} (1960), 423--442. \MR{0111939 (22 \#2797)}

\bibitem[AS98]{AgranovichSobolevskii98}
Y~Y. Agranovich and P.~E. Sobolevskii, \emph{Motion of nonlinear visco-elastic
  fluid}, Nonlinear Anal. \textbf{32} (1998), no.~6, 755--760. \MR{1612126
  (99g:76008)}

\bibitem[BC04]{Blakely_Corron04}
J.~N. Blakely and N.~J. Corron, \emph{Experimental observation of delay-induced
  radio frequency chaos in a transmission line oscillator}, Chaos \textbf{14}
  (2004), 1035.

\bibitem[BCFM95]{BCFM95}
H.~Bercovici, P.~Constantin, C.~Foias, and O.~P. Manley, \emph{Exponential
  decay of the power spectrum of turbulence}, J. Statist. Phys. \textbf{80}
  (1995), no.~3-4, 579--602. \MR{1342242 (97b:76071)}

\bibitem[BDS55]{Bartle_al55}
R.~G. Bartle, N.~Dunford, and J.~Schwartz, \emph{Weak compactness and vector
  measures}, Canad. J. Math. \textbf{7} (1955), 289--305. \MR{0070050
  (16,1123c)}

\bibitem[Bil99]{Billingsley1}
P.~Billingsley, \emph{Convergence of probability measures}, second ed., Wiley
  Series in Probability and Statistics: Probability and Statistics, John Wiley
  \& Sons Inc., New York, 1999, A Wiley-Interscience Publication. \MR{MR1700749
  (2000e:60008)}

\bibitem[BK84]{BrunovskyKomornik1984}
P.~Brunovsk{{\'y}} and J.~Komorn{\'{\i}}k, \emph{Ergodicity and exactness of
  the shift on {$C[0,\infty)$} and the semiflow of a first-order partial
  differential equation}, J. Math. Anal. Appl. \textbf{104} (1984), no.~1,
  235--245. \MR{765054 (86m:58084)}

\bibitem[BKW07]{Barton_al07}
D.~A.~W. Barton, B.~Krauskopf, and R.~E. Wilson, \emph{Homoclinic bifurcations
  in a neutral delay model of a transmission line oscillator}, Nonlinearity
  \textbf{20} (2007), no.~4, 809--829. \MR{2307881 (2008b:34137)}

\bibitem[BM64]{Brayton_Miranker64}
R.~K. Brayton and W.~L. Miranker, \emph{A stability theory for nonlinear mixed
  initial boundary value problems}, Arch. Rational Mech. Anal. \textbf{17}
  (1964), 358--376. \MR{0168864 (29 \#6120)}

\bibitem[CCP]{ChepyzhovContiPata2011}
V.~V. Chepyzhov, M.~Conti, and V.~Pata, \emph{A minimal approach to the theory
  of global attractors}, (to appear).

\bibitem[CDPGHP11]{ChekrounDiPlinioGlattHoltzPata2010}
M.~D. Chekroun, F.~Di~Plinio, N.~E. Glatt-Holtz, and V.~Pata, \emph{Asymptotics
  of the {C}oleman-{G}urtin model}, Discrete Contin. Dyn. Syst. Ser. S
  \textbf{4} (2011), no.~2, 351--369. \MR{2746378}

\bibitem[CG67]{ColemanGurtin1}
B.~D. Coleman and M.~E. Gurtin, \emph{Equipresence and constitutive equations
  for rigid heat conductors}, Z. Angew. Math. Phys. \textbf{18} (1967),
  199--208. \MR{0214334 (35 \#5185)}

\bibitem[CG06]{CG06}
C.~Cavaterra and M.~Grasselli, \emph{Robust exponential attractors for
  population dynamics models with infinite time delay}, Discrete Contin. Dyn.
  Syst. Ser. B \textbf{6} (2006), no.~5, 1051--1076 (electronic). \MR{2224870
  (2007d:35132)}

\bibitem[CGG{\etalchar{+}}06]{ChepyzhovGattiGrasselliMiranvillePata}
V.~V. Chepyzhov, S.~Gatti, M.~Grasselli, A.~Miranville, and V.~Pata,
  \emph{Trajectory and global attractors for evolution equations with memory},
  Appl. Math. Lett. \textbf{19} (2006), no.~1, 87--96. \MR{2189821
  (2006h:37109)}

\bibitem[CGHa]{ChekrounGlattHoltz2011b}
M.~D. Chekroun and N.~Glatt-Holtz, \emph{Invariant measures for dissipative
  non-autonomous dynamical systems}, (in preparation).

\bibitem[CGHb]{ChekrounGlattHoltz2011a}
\bysame, \emph{The stochastic navier-stokes equations with memory: Random
  attractors and invariant measures}, (in preparation).

\bibitem[Chu99]{Chueshov99}
I.~D. Chueshov, \emph{Introduction to the theory of infinite dimensional
  dissipative systems}, University Lectures in Contemporary Mathematics, AKTA,
  Kharkiv, 1999. \MR{1788405 (2001k:37126)}

\bibitem[CK70]{Cruz_Hale70}
M.~Cruz, A. and Hale~J. K., \emph{Stability of functional differential
  equations of neutral type}, J. Differential Equations \textbf{7} (1970),
  334--355. \MR{0257516 (41 \#2166)}

\bibitem[CP06]{ChepyzhovPata}
V.~V. Chepyzhov and V.~Pata, \emph{Some remarks on stability of semigroups
  arising from linear viscoelasticity}, Asymptot. Anal. \textbf{46} (2006),
  no.~3-4, 251--273. \MR{MR2215885 (2007c:47053)}

\bibitem[CPS06]{ContiPataSquassina}
M.~Conti, V.~Pata, and M.~Squassina, \emph{Singular limit of differential
  systems with memory}, Indiana Univ. Math. J. \textbf{55} (2006), no.~1,
  169--215. \MR{MR2207550 (2006k:35290)}

\bibitem[CSG11]{CSG11}
M.D. Chekroun, E.~Simonnet, and M.~Ghil, \emph{Stochastic climate dynamics:
  Random attractors and time-dependent invariant measures}, Physica D (2011),
  no.~240, 1685--1700.

\bibitem[CV02]{ChepyzhovVishik2002}
V.~V. Chepyzhov and M.I. Vishik, \emph{Attractors for equations of mathematical
  physics}, American Mathematical Society Colloquium Publications, vol.~49,
  American Mathematical Society, Providence, RI, 2002.

\bibitem[Daf70]{Dafermos1}
C.~M. Dafermos, \emph{Asymptotic stability in viscoelasticity}, Arch. Rational
  Mech. Anal. \textbf{37} (1970), 297--308. \MR{MR0281400 (43 \#7117)}

\bibitem[Daw83]{Dawidowicz1983}
A.~L. Dawidowicz, \emph{On the existence of an invariant measure for the
  dynamical system generated by partial differential equation}, Ann. Polon.
  Math. \textbf{41} (1983), no.~2, 129--137. \MR{724128 (85b:58073)}

\bibitem[Deb11]{Debussche2011a}
A.~Debussche, \emph{Ergodicity results for the stochastic navier-stokes
  equations: an introduction.}, (to appear).

\bibitem[DPD03]{DaPratoDebussche}
G.~Da~Prato and A.~Debussche, \emph{Ergodicity for the 3{D} stochastic
  {N}avier-{S}tokes equations}, J. Math. Pures Appl. (9) \textbf{82} (2003),
  no.~8, 877--947. \MR{MR2005200 (2004m:60133)}

\bibitem[DPP08]{DiPlinioPata1}
F.~Di~Plinio and V.~Pata, \emph{Robust exponential attractors for the strongly
  damped wave equation with memory. {I}}, Russ. J. Math. Phys. \textbf{15}
  (2008), no.~3, 301--315. \MR{MR2448344}

\bibitem[DPPZ08]{DiPlinioPata2}
F.~Di~Plinio, V.~Pata, and S.~Zelik, \emph{On the strongly damped wave equation
  with memory}, Indiana Univ. Math. J. \textbf{57} (2008), no.~2, 757--780.
  \MR{MR2414334}

\bibitem[DPZ96]{ZabczykDaPrato2}
G.~Da~Prato and J.~Zabczyk, \emph{Ergodicity for infinite-dimensional systems},
  London Mathematical Society Lecture Note Series, vol. 229, Cambridge
  University Press, Cambridge, 1996. \MR{MR1417491 (97k:60165)}

\bibitem[Dug51]{Dugundji1951}
J.~Dugundji, \emph{An extension of {T}ietze's theorem}, Pacific J. Math.
  \textbf{1} (1951), 353--367. \MR{0044116 (13,373c)}

\bibitem[FM95]{FlandoliMaslowski1}
F.~Flandoli and B.~Maslowski, \emph{Ergodicity of the {$2$}-{D}
  {N}avier-{S}tokes equation under random perturbations}, Comm. Math. Phys.
  \textbf{172} (1995), no.~1, 119--141. \MR{MR1346374 (96g:35223)}

\bibitem[FMRT01]{FoiasManleyRosaTemam1}
C.~Foias, O.~Manley, R.~Rosa, and R.~Temam, \emph{Navier-{S}tokes equations and
  turbulence}, Encyclopedia of Mathematics and its Applications, vol.~83,
  Cambridge University Press, Cambridge, 2001. \MR{MR1855030 (2003a:76001)}

\bibitem[FS86]{FrancfortSuquet1986}
G.~A. Francfort and P.~M. Suquet, \emph{Homogenization and mechanical
  dissipation in thermoviscoelasticity}, Arch. Rational Mech. Anal. \textbf{96}
  (1986), no.~3, 265--293. \MR{855306 (88a:73012)}

\bibitem[FT75]{FoiasTemam1975}
C.~Foias and R.~Temam, \emph{On the stationary statistical solutions of the
  navier-stokes equations and turbulence}, Publications Mathematiques D'Orsay
  (1975), no.~120-75-28.

\bibitem[GGMP06]{GattiGrasselliMiranvillePata}
S.~Gatti, M.~Grasselli, A.~Miranville, and V.~Pata, \emph{Memory relaxation of
  the one-dimensional {C}ahn-{H}illiard equation}, Dissipative phase
  transitions, Ser. Adv. Math. Appl. Sci., vol.~71, World Sci. Publ.,
  Hackensack, NJ, 2006, pp.~101--114. \MR{2223375 (2006m:35148)}

\bibitem[GGP99]{GiorgiGrasselliPata99}
C.~Giorgi, M.~Grasselli, and V.~Pata, \emph{Uniform attractors for a
  phase-field model with memory and quadratic nonlinearity}, Indiana Univ.
  Math. J. \textbf{48} (1999), no.~4, 1395--1445. \MR{1757078 (2001h:37160)}

\bibitem[GGP04]{GattiGrasselliPata2004}
S.~Gatti, M.~Grasselli, and V.~Pata, \emph{Exponential attractors for a
  phase-field model with memory and quadratic nonlinearities}, Indiana Univ.
  Math. J. \textbf{53} (2004), no.~3, 719--753. \MR{2086698 (2005e:37184)}

\bibitem[GGP05]{GattiGiorgiPata}
S.~Gatti, C.~Giorgi, and V.~Pata, \emph{Navier-{S}tokes limit of {J}effreys
  type flows}, Phys. D \textbf{203} (2005), no.~1-2, 55--79. \MR{MR2135133
  (2006a:37081)}

\bibitem[GMLB{\etalchar{+}}08]{GMBDC08}
P.~Gouze, Y.~Melean, T.~Le~Borgne, M.~Dentz, and J.~Carrera, \emph{Non-fickian
  dispersion in porous media explained by heterogeneous microscale matrix
  diffusion}, Water Resources Research \textbf{44} (2008), no.~W11416, 1--19.

\bibitem[GMPZ08]{GattiMiranvillePataZelik2008}
S.~Gatti, A.~Miranville, V.~Pata, and S.~Zelik, \emph{Attractors for
  semi-linear equations of viscoelasticity with very low dissipation}, Rocky
  Mountain J. Math. \textbf{38} (2008), no.~4, 1117--1138. \MR{2436716
  (2009g:35320)}

\bibitem[GP]{GrasselliPata}
M.~Grasselli and V.~Pata, \emph{Uniform attractors of nonautonomous systems
  with memory}.

\bibitem[GP68]{GurtinPipkin1}
M.~E. Gurtin and A.~C. Pipkin, \emph{A general theory of heat conduction with
  finite wave speeds}, Arch. Rational Mech. Anal. \textbf{31} (1968), no.~2,
  113--126. \MR{1553521}

\bibitem[GP02]{GrasselliPata2002}
M.~Grasselli and V.~Pata, \emph{Uniform attractors of nonautonomous dynamical
  systems with memory}, Evolution equations, semigroups and functional analysis
  ({M}ilano, 2000), Progr. Nonlinear Differential Equations Appl., vol.~50,
  Birkh{\"a}user, Basel, 2002, pp.~155--178. \MR{1944162 (2003j:37135)}

\bibitem[GP05]{GrasselliPata1}
\bysame, \emph{Attractors of phase-field systems with memory}, Mathematical
  methods and models in phase transitions, Nova Sci. Publ., New York, 2005,
  pp.~157--175. \MR{2590949 (2010m:80007)}

\bibitem[GP06]{GrasselliPata2}
\bysame, \emph{A reaction-diffusion equation with memory}, Discrete Contin.
  Dyn. Syst. \textbf{15} (2006), no.~4, 1079--1088. \MR{2224498 (2006m:35184)}

\bibitem[GT00]{GT00}
E.~Galanti and E.~Tziperman, \emph{Enso's phase locking to the seasonal cycle
  in the fast sst, fast wave, and mixed mode regimes}, J. Atmospheric Sci.
  \textbf{57} (2000), 2936--2950.

\bibitem[Hal77]{Hale77}
J.~K. Hale, \emph{Theory of functional differential equations}, second ed.,
  Springer-Verlag, New York, 1977, Applied Mathematical Sciences, Vol. 3.
  \MR{0508721 (58 \#22904)}

\bibitem[Hal88]{Hale1988}
\bysame, \emph{Asymptotic behavior of dissipative systems}, Mathematical
  Surveys and Monographs, vol.~25, American Mathematical Society, Providence,
  RI, 1988. \MR{941371 (89g:58059)}

\bibitem[HC69]{Cruz_Hale69}
J.~K. Hale and M.~A. Cruz, \emph{Asymptotic behavior of neutral functional
  differential equations}, Arch. Rational Mech. Anal. \textbf{34} (1969),
  331--353. \MR{0249760 (40 \#3001)}

\bibitem[HL93]{HaleLunel93}
J.~K. Hale and S.~M.~V. Lunel, \emph{Introduction to functional differential
  equations}, Springer-Verlag, 1993.

\bibitem[HM06]{HairerMattingly1}
M.~Hairer and J.~C. Mattingly, \emph{Ergodicity of the 2{D} {N}avier-{S}tokes
  equations with degenerate stochastic forcing}, Ann. of Math. (2) \textbf{164}
  (2006), no.~3, 993--1032. \MR{MR2259251 (2008a:37095)}

\bibitem[HM11]{HairerMattingly2011}
\bysame, \emph{A theory of hypoellipticity and unique ergodicity for semilinear
  stochastic pdes}, Electron. J. Probab. \textbf{16} (2011), no.~23, 658--738.

\bibitem[JN93a]{JN93a}
F.-F. Jin and J.D. Neelin, \emph{Modes of interannual tropical
  ocean--atmosphere interaction---a unified view. part i: Numerical results.},
  J. Atmospheric Sci. \textbf{50} (1993), 3477--3503.

\bibitem[JN93b]{JN93b}
\bysame, \emph{Modes of interannual tropical ocean--atmosphere interaction---a
  unified view. part iii: Analytical results in fully coupled cases.}, J.
  Atmospheric Sci. \textbf{50} (1993), 3523--3540.

\bibitem[Jos90]{Joseph90}
D.~D. Joseph, \emph{Fluid dynamics of viscoelastic liquids}, Applied
  Mathematical Sciences, vol.~84, Springer-Verlag, New York, 1990. \MR{1051193
  (91d:76003)}

\bibitem[KH10]{KH10}
Y.~N. Kyrychko and S.~J. Hogan, \emph{On the use of delay equations in
  engineering applications}, J. Vibration and Control \textbf{16} (2010),
  943---960.

\bibitem[KS01]{KuksinShirikyan1}
S.~Kuksin and A.~Shirikyan, \emph{A coupling approach to randomly forced
  nonlinear {PDE}'s. {I}}, Comm. Math. Phys. \textbf{221} (2001), no.~2,
  351--366. \MR{1845328 (2002e:35253)}

\bibitem[Lad91]{Ladyzhenskaya91}
O.~Ladyzhenskaya, \emph{Attractors for semigroups and evolution equations},
  Lezioni Lincee. [Lincei Lectures], Cambridge University Press, Cambridge,
  1991. \MR{1133627 (92k:58040)}

\bibitem[Las79]{Lasota1979}
A.~Lasota, \emph{Invariant measures and a linear model of turbulence}, Rend.
  Sem. Mat. Univ. Padova \textbf{61} (1979), 39--48 (1980). \MR{569650
  (81d:54032)}

\bibitem[Lax02]{Lax2002}
P.~D. Lax, \emph{Functional analysis}, Pure and Applied Mathematics (New York),
  Wiley-Interscience [John Wiley \& Sons], New York, 2002. \MR{1892228
  (2003a:47001)}

\bibitem[LRR11]{LukaszewiczRealRobinson2011}
G.~Lukaszewicz, J.~Real, and J.~C. Robinson, \emph{Invariant measures for
  dissipative systems and generalised banach limits}, Journal of Dynamics and
  Differential Equations \textbf{23} (2011), no.~2, 225--250.

\bibitem[MN07]{Michiels_Niculescu}
W.~Michiels and S-I. Niculescu, \emph{Stability and stabilization of time-delay
  systems}, Advances in Design and Control, vol.~12, Society for Industrial and
  Applied Mathematics (SIAM), Philadelphia, PA, 2007, An eigenvalue-based
  approach. \MR{2384531 (2009a:93001)}

\bibitem[MWZ02]{Ma_Wang_Zhong}
Q.~Ma, S.~Wang, and C.~Zhong, \emph{Necessary and sufficient conditions for the
  existence of global attractors for semigroups and applications}, Indiana
  Univ. Math. J. \textbf{51} (2002), no.~6, 1541--1559. \MR{1948459
  (2003j:37137)}

\bibitem[MZ08]{MiranvilleZelik2008}
A.~Miranville and S.~Zelik, \emph{Attractors for dissipative partial
  differential equations in bounded and unbounded domains}, Handbook of
  differential equations: evolutionary equations. {V}ol. {IV}, Handb. Differ.
  Equ., Elsevier/North-Holland, Amsterdam, 2008, pp.~103--200. \MR{2508165
  (2010c:37175)}

\bibitem[Nic01]{Niculescu_stab}
S.-I. Niculescu, \emph{On robust stability of neutral systems}, Kybernetika
  (Prague) \textbf{37} (2001), no.~3, 253--263, Special issue on advances in
  analysis and control of time-delay systems. \MR{1859084 (2002h:93084)}

\bibitem[Orl99]{Orlov99}
V.~P. Orlov, \emph{On the {O}ldroyd model of a viscoelastic fluid},
  Funktsional. Anal. i Prilozhen. \textbf{33} (1999), no.~1, 83--87.
  \MR{1711831 (2001b:35249)}

\bibitem[Pao97]{Pao97}
C.~V. Pao, \emph{Systems of parabolic equations with continuous and discrete
  delays}, J. Math. Anal. Appl. \textbf{205} (1997), no.~1, 157--185.
  \MR{1426986 (97j:35155)}

\bibitem[PZ01]{PataZucchi1}
V.~Pata and A.~Zucchi, \emph{Attractors for a damped hyperbolic equation with
  linear memory}, Adv. Math. Sci. Appl. \textbf{11} (2001), no.~2, 505--529.
  \MR{1907454 (2003f:35027)}

\bibitem[RHN87]{RHN87}
M.~Renardy, W.~J. Hrusa, and J.~A. Nohel, \emph{Mathematical problems in
  viscoelasticity}, Pitman Monographs and Surveys in Pure and Applied
  Mathematics, vol.~35, Longman Scientific \& Technical, Harlow, 1987.
  \MR{919738 (89b:35134)}

\bibitem[Rob01]{Robinson1}
J.~Robinson, \emph{Infinite-dimensional dynamical systems}, Cambridge Texts in
  Applied Mathematics, Cambridge University Press, Cambridge, 2001, An
  introduction to dissipative parabolic PDEs and the theory of global
  attractors. \MR{MR1881888 (2003f:37001a)}

\bibitem[Ros98]{Rosa98}
R.~Rosa, \emph{The global attractor for the {$2$}{D} {N}avier-{S}tokes flow on
  some unbounded domains}, Nonlinear Anal. \textbf{32} (1998), no.~1, 71--85.
  \MR{1491614 (98k:35152)}

\bibitem[Rud87]{Rudin1987}
W.~Rudin, \emph{Real and complex analysis}, third ed., McGraw-Hill Book Co.,
  New York, 1987. \MR{924157 (88k:00002)}

\bibitem[Rud88]{Rudnicki1988}
R.~Rudnicki, \emph{Strong ergodic properties of a first-order partial
  differential equation}, J. Math. Anal. Appl. \textbf{133} (1988), no.~1,
  14--26. \MR{949314 (90a:58097)}

\bibitem[Rud04]{Rudnicki2004}
\bysame, \emph{Chaos for some infinite-dimensional dynamical systems}, Math.
  Methods Appl. Sci. \textbf{27} (2004), no.~6, 723--738. \MR{2070224
  (2005c:37149)}

\bibitem[Smi64]{Smirnov}
V.I. Smirnov, \emph{A course of higher mathematics}, vol.~2, Pergamon Press,
  Oxford, 1964.

\bibitem[Smi11]{Smith11}
H.~Smith, \emph{An introduction to delay differential equations with
  applications to the life sciences}, Texts in Applied Mathematics, vol.~57,
  Springer, New York, 2011. \MR{2724792 (2011k:34002)}

\bibitem[Tem97]{Temam3}
R.~Temam, \emph{Infinite-dimensional dynamical systems in mechanics and
  physics}, second ed., Applied Mathematical Sciences, vol.~68,
  Springer-Verlag, New York, 1997. \MR{MR1441312 (98b:58056)}

\bibitem[Tem01]{Temam1}
\bysame, \emph{Navier-{S}tokes equations: Theory and numerical analysis}, AMS
  Chelsea Publishing, Providence, RI, 2001, Reprint of the 1984 edition.
  \MR{MR1846644 (2002j:76001)}

\bibitem[Wan09]{Wang2009}
X.~Wang, \emph{Upper semi-continuity of stationary statistical properties of
  dissipative systems}, Discrete Contin. Dyn. Syst. \textbf{23} (2009),
  no.~1-2, 521--540. \MR{2449091 (2010h:37192)}

\bibitem[Wu96]{Wu96}
J.~Wu, \emph{Theory and applications of partial functional-differential
  equations}, Applied Mathematical Sciences, vol. 119, Springer-Verlag, New
  York, 1996. \MR{1415838 (98a:35135)}

\bibitem[ZT86]{ZT86}
E.~C. Zachmanoglou and D.~W. Thoe, \emph{Introduction to partial differential
  equations with applications}, second ed., Dover Publications Inc., New York,
  1986. \MR{880021 (88a:35001)}

\end{thebibliography}

\normalsize

\newpage

\noindent Micka\"el D. Chekroun \\
{\footnotesize
Department of Atmospheric Sciences and Institute of Geophysics and Planetary Physics, \\
University of California, Los Angeles, CA 90095-1565, USA \\
and \\
Environmental Research and Teaching Institute (CERES-ERTI), \\
\'Ecole Normale Sup\'erieure, 75231 Paris Cedex 05, France \\
Web: \url{http://www.environnement.ens.fr/annuaire/chekroun-mickael/}\\
Email: \url{mchekroun@atmos.ucla.edu}}\\[.3cm]
\noindent Nathan E. Glatt-Holtz\\ {\footnotesize
Department of Mathematics
and The Institute for Scientific Computing and Applied Mathematics\\
Indiana University\\
Web: \url{http://mypage.iu.edu/\~negh/}\\
 Email: \url{negh@indiana.edu}} \\[.3cm]

\end{document}